\documentclass{amsart}
\usepackage{amssymb,amsmath,latexsym,amsthm}
\usepackage[english]{babel}

\usepackage{amsmath,amssymb,amsbsy,amsfonts,amsthm,latexsym,
                       amsopn,amstext,amsxtra,euscript,amscd}
\usepackage{hyperref}
\usepackage{array,color}
\usepackage{ifthen}
\usepackage{url}

\newtheorem{thm}{Theorem}

\newtheorem{lem}[thm]{Lemma}

\newtheorem{rmq}[thm]{Remark}

\begin{document}

\title[quantitative growth of linear recurrences]{ quantitative growth of linear recurrences}

\author[Armand Noubissie]{Armand Noubissie}
\address{
{Graz University of Technology}\newline
{Institute for Analysis and Number Theory}\newline
{M\"{u}nzgrabenstrasse 36/II, 8010 Graz, Austria}}
\email{\tt armand.noubissie@tugraz.at}


\date{\today}

\maketitle

\begin{abstract}
Let $\{u_n\}_n$ be a non-degenerate linear recurrence sequence of integers with Binet's formula given by $u_n= \sum_{i=1}^{m} P_i(n)\alpha_i^n.$  Assume $\max_i \vert \alpha_i \vert >1$. In 1977, Loxton and Van der Poorten  conjectured that for any $\epsilon >0$ there is a effectively computable constant $C(\epsilon),$ such that if $ \vert u_n \vert < (\max_i\{ \vert \alpha_i  \vert \})^{n(1-\epsilon)}$, then $n<C(\epsilon)$. Using results of Schmidt and  Evertse, a complete non-effective (qualitative) proof of this conjecture was given by  Fuchs and Heintze (2021) and, 
independently, by Karimov  and al.~(2023). In this paper, we give an effective upper bound for the number of solutions of 
the inequality $\vert u_n \vert < (\max_i\{ \vert \alpha_i  \vert \})^{n(1-\epsilon)}$, 
thus extending several earlier results by
Schmidt, Schlickewei and Van der Poorten.

\end{abstract}
\section{Introduction}\label{sec1}
 A  sequence of integer numbers $\{u_n\}$ is a linear recurrence sequence  of order $l$  if it is defined by the recursion  relation
 \begin{equation} \label{eq:a}
 u_{n+l}= a_1u_{n+l-1}+ a_2u_{n+l-2}+ \cdots  + a_lu_{n}, \quad (n \in \mathbb{N}),
 \end{equation} 
where the coefficients $a_1, \cdots, a_l$ are complex numbers and $l$  minimal. It is well-known that we can write
$$
 u_n = \sum_{i=1}^{m} P_i(n) \alpha_i^n, 
$$
where $\alpha_i$ are the roots of the characteristic polynomial $P$ of $\{u_n\}$ and the coefficients of polynomials $P_i$ belong to $K:= \mathbb{Q}(\alpha_1, \cdots, \alpha_m)$ and $d$ the degree of the number field $K$ over $\mathbb{Q}$. One says that such a sequence is non-degenerate if none of the ratios $\alpha_i/\alpha_j ~(i \neq j)$  is a root of unity. Let  $\alpha_1$  be a dominant root of $P$. We denote by $s$  the cardinality of the set  containing all the prime ideals above $\alpha_i's$ and  all the archimedean  places over $K$. Here, we are interested in the study of the exponential Diophantine equation $u_n=0$.  Since any linear recurrence sequence can be decomposed as the interleaving of non-degenerate sequences, we study the previous equation only for the non-degenerate sequences. This problem was first studied by Skolem, Mahler, Lech, who used $p$-adic analysis to prove that this equation has finitely many solutions.  Their proof is  ineffective  in the sense that it does not allow one to bound the number of solutions. Van der Poorten and Schlickewei \cite{VS} gave an upper bound for the number of 
solutions  depending only  on $d, l$ and $s$. This result was extended by Schlickewei and Schmidt (\cite[Thm.\,2.1]{SS}), who provided an upper bound depending only on $d$ and $l$ and Schmidt \cite{S} an upper bound depending only on $l$.\\

We now turn to a more  general question concerning the rate of growth of the recurrence. It is not difficult to see  there is an effectively computable
constant $C$ such that, for all
$n \geq 1, ~~ \vert u_n \vert < C n^a \vert \alpha_1 \vert^n,$ where $\alpha_1$ is a dominant root of the characteristic polynomial of  sequence $\{u_n\}$ with $a$,  the maximum of  the multiplicity of $\alpha_i's$.
In 1977, Loxton and Van der Poorten \cite{LV} conjectured 
that any non-degenerate LRS (linear recurrence sequences) has essentially, the maximal possible
growth rate, i.e., for any $\epsilon >0$ there is an effectively computable constant $C(\epsilon),$ such that if $ \vert u_n \vert < (\max_i\{ \vert \alpha_i  \vert \})^{n(1-\epsilon)}$, then $n<C(\epsilon)$. Using results of Schmidt \cite{S} and Evertse \cite{E1}, a complete non-effective (qualitative) proof of this conjecture was given by  Fuchs and Heintze \cite{FH} and, 
independently, by Karimov  and al. \cite{K}.  \\

This question about the rate of growth of LRS is connected to the
Positivity and Ultimate Positivity Problems (in theoretical computer science). In fact Ouaknine and Worrell \cite{OW3}   have proven that for simple linear recurrence sequence,  the Ultimate Positivity  Problem is decidable for all order. Their algorithm arises from model theory and the qualitative version of Loxton - Van der poorten  conjecture above. The non-effectiveness of this result makes that algorithm non-constructive in the sense that, for a given linear recurrence sequence $\{u_n\}$, the procedure does not produce a threshold $N$ such that $u_n \geq 0$ for all $n \geq N$. This is why the positivity problem for simple linear recurrence sequence is still open.\\

In this article, we make a progress in that direction by giving an explicit  upper bound for the number of solutions. \\
Let $\{u_n\}$ be the  non-degenerate linear recurrence sequence of integers  of order $l \geq 1$ defined by the recursion relation \eqref{eq:a} with Binet's formula given by  $$
 u_n = \sum_{i=1}^{m} P_i(n) \alpha_i^n, 
$$
where $\alpha_i$ are the roots of the characteristic polynomial $P$ of $\{u_n\}$ with multiplicity at most $a$;  the coefficients $a_{ij}$ of polynomials $P_i$ belong to $K:= \mathbb{Q}(\alpha_1, \cdots, \alpha_m)$ and $d$ the degree of the number field $K$ over $\mathbb{Q}$.  Assume $\alpha_1$  is a dominant root of $P$ and $\vert \alpha_1 \vert >1$. We denote by $S$   the set  containing all the prime ideals above $\alpha_i's$ and  all the archimedean  places over $K$ and $s$ its cardinality.   Let  $n_{ij}$ be the smallest positive integer such that $n_{ij}a_{ij}  \in \mathcal{O}_K$  and we denote by $\rho := {\rm lcm}(n_{ij})$.  Put
\begin{equation} \label{eq:b}
   A':= \min_{\sigma, i, j} \{ 1, ~\vert \sigma(\rho a_{ij})\vert,~~\sigma(a_{ij}) \neq 0  \}, \quad A := \max_{\sigma, i,j} \{ 1, ~\vert \sigma(\rho a_{ij})\vert \} 
\end{equation}  
   where the minimum and the maximum is taken respectively into all elements of ${\rm Gal}(K/ \mathbb{Q})$  and coefficients of $\rho P_i$.   Let $\Delta_K$ be the discriminant of $K$. For  $x \in \mathbb{R}^{*}$, we define the functions $g(x)=2^{34}\cdot 2^{11l}s^2l^{12}x^{-5}$ and 
 
 $$
\tau(x) = 
	 \max \biggl\{ \dfrac{20a^lA^l}{ A'} ,~ \dfrac{-\log(A')}{\epsilon \log(\vert \alpha_1 \vert)},~ \dfrac{11mA + e^{g(x)}\log(l!)\log(2d\vert \Delta_K \vert)}{x\log \vert \alpha_1 \vert}  \biggl\}.
$$
Notice that $\rho =1$ is equivalent to, all the coefficients of the polynomials $P_i$ are algebraic integers. Hence our main result can be state as following 
\\
\begin{thm} \label{thm}
For each $0<\epsilon < 1/ 12,$  the number of solutions of the inequality 
\begin{equation} \label{eq:1}
  \vert u_n \vert <  \vert \alpha_1  \vert ^{n(1-\epsilon)} 
 \end{equation} 
 does not exceed                                                     \begin{itemize}
 \item  $   \tau(\epsilon / 2d) + 2\cdot s^2 \left(  d^2l^{7} 2^{4l+19} \epsilon ^{-2}\right)^{ls+l} \exp \left( (7l^a)^{8l^a}\right),$ if $\rho =1$ and \\
 \item  $ \dfrac{2\epsilon^{-1} \log \rho}{\log \vert \alpha_1 \vert}  +  \tau(\epsilon / 4d) + 2\cdot  s^2 \left( d^2 l^{7} 2^{4l+21} \epsilon ^{-2}\right)^{ls+l} \exp \left( (7l^a)^{8l^a}\right),$  if  $\rho \neq 1.$ 	
 \end{itemize}
\end{thm}
The following result due by Schmidt \cite{S} plays a crucial role in the proof of Theorem \ref{thm}.
\begin{lem}\label{lemm}
Suppose that $\{u_n\}_{n \in \mathbb{Z}}$ is a non-degenerate linear recurrence sequence of complex numbers, whose characteristic polynomial has $k$ distinct roots of  multiplicity at most $a$. Then the number of solutions $n \in Z$ of the equation $u_n=0$ can be bounded above by $$c(k, a) = \exp((7k^a)^{8k^a}).$$
\end{lem}

\begin{rmq}
An effective upper bound for the number of solutions of inequality \eqref{eq:1}  can be obtained by "carefully" combining a variant of the main Theorem in \cite{E} and Lemma \ref{lemm}. However, our Theorem \ref{thm} provide a better upper bound.
\end{rmq}
If the characteristic polynomial of $\{u_n\}_n$ has just one root i.e. $m=1$,  then the inequality  \eqref{eq:1} becomes $$ \vert P_1(n)\alpha_1^n \vert  < ( \vert \alpha_1  \vert )^{n(1-\epsilon)} $$ which implies $ \vert P_1(n)\alpha_1^{n\epsilon} \vert  < 1$. If $n  >  \dfrac{2(aA+1)}{A'},$ where $A$ and $A'$ are defined in \eqref{eq:b},  then one get $\vert P_i(n) \vert >  A'$. So  $A' \vert \alpha_1^{n\epsilon} \vert  < 1$. Hence we conclude that any non-negative integer $n$ solution of the inequality  \eqref{eq:1} satisfy the relation
$$n <  \max \left( \dfrac{2(aA+1)}{A'} , ~ \dfrac{-\log(A')}{\epsilon \log(\vert \alpha_1 \vert)} \right).$$

 In the rest of the article we assume $m>1$.
  Let  $M_{K}$ be the collection of places over $K$. For $v \in M_K$,  $x \in K$, we define $\vert x \vert_v$ as follows:
  \begin{itemize}
  \item $\vert x \vert_v = \vert \sigma (x) \vert ^{1/d}$ if $v$ corresponds to the embedding  $\sigma\,:\, K \rightarrow  \mathbb{R}$\\
   \item $\vert x \vert_v = \vert \sigma (x) \vert^{2/d}$ if $v$ corresponds to the embedding  $\sigma\,:\, K \rightarrow  \mathbb{C}$\\
   \item $\vert x \vert_v = (N_{K/\mathbb{Q}}(\mathcal{P}))^{-{\rm ord}_{\mathcal{P}}(x)/d}$ if $v$ corresponds to the prime ideal $\mathcal{P}$ of  $\mathcal{O}_K$ and  ${\rm ord}_{\mathcal{P}}(x)$ the exponent of  $\mathcal{P}$ in the decomposition of the ideal generated by $x$. 
 \end{itemize}
 
 With this norm, we have the product formula $$ \prod_{ v \in M_K}^{} \vert x \vert_v =1, \quad  \mbox{for} \quad  x \in K.$$ For $X=(x_1, \cdots, x_m) \in K^m$, $X \neq 0$ and $\vert X \vert_v = \max \left( \vert x_1\vert_v, \cdots, \vert x_m \vert_v \right)$, we define the height of $X$  as follows:
  $$H(X)= \prod_{ v \in M_K}^{} \vert X \vert_v .$$ Note that $H(X)$ depends only on $X$ and not on the choice of the number field $K$ (see Lemma 1.5.2 in \cite{BV}). Let $K$ be the splitting field of the characteristic polynomial of $\{u_n\}$ and $S$ be the set  containing all the prime ideals above $\alpha_i's$ and  all the archimedean  places over $K$ and $s$ its cardinality.  The ring of $S$-integers and its set of unit are respectively defined as  $$\mathcal{O}_S= \{x \in K~:~ \vert x \vert_v \leq 1 \quad \mbox{for} ~~\mbox{all} \quad v\notin S\}$$ and $$\mathcal{O}_S^*= \{x \in K~:~ \vert x \vert_v = 1 \quad \mbox{for} ~~\mbox{all} \quad v\notin S\}.$$  We denote  by ${\rm id}$  the embedding over $K$ corresponding to the identity and for $v \in M_K$, the quantity $s(v)$ by
   \begin{equation*}
s(v) = \left\{
\begin{aligned}
	  \dfrac{1}{d}  & \quad \mbox{if}\quad v ~~ \mbox{is} ~~\mbox{real}~~\mbox{infinite}, \\\\
	\dfrac{2}{d}  & \quad \mbox{if}\quad v ~~ \mbox{is} ~~\mbox{complex}~~\mbox{infinite},\\\\
	0  & \quad \mbox{if}\quad v ~~ \mbox{is} ~~\mbox{finite},
\end{aligned}
\right.
\end{equation*}
 where $d$ is the degree of $K$ over $\mathbb{Q}$.
By the definition of $s(v)$, we get 
\begin{equation} \label{eq:d}
\sum_{v \in M_K}^{} s(v) = 1.
\end{equation}
For $v \in S \setminus \{{\rm id}\}$, we define $m$ linear 
  forms $l_{1v}, \cdots, l_{mv}$ with $m$ variables $x:=(x_1, \cdots, x_m)$ over $K$ as  $l_{iv}(x)=x_i$. For $v={\rm id},$ we define $l_{1v} (x) = x_1+\cdots + x_m$ and $l_{iv}(x)=x_i$ for $i=2,\cdots, m.$ 
  
 Notice that, to proving Theorem \ref{thm}, it  suffice to show only the first part (i.e. case $\rho =1$ meaning all the coefficients of the polynomials $P_i$ are algebraic integers)  since the second part follows easily from the first one.  Indeed, assume $\rho \neq 1$ and fix $\epsilon >0$. It is clear that all the solutions of inequality \eqref{eq:1} are solutions of the inequality   $$\vert \rho u_n \vert < \dfrac{\rho}{\vert \alpha_1  \vert ^{n\epsilon / 2} }\vert \alpha_1  \vert ^{n(1-\epsilon/2)}.$$ If $\rho > \vert \alpha_1  \vert ^{n\epsilon / 2}$ then we obtain $$n< \dfrac{2\epsilon^{-1} \log \rho}{\log \vert \alpha_1 \vert}.$$ Otherwise, one obtains $\vert \rho u_n \vert < (\vert \alpha_1  \vert )^{n(1-\epsilon/2)}$.  Since  the coefficients of the polynomials $\rho P_i$ in the Binet's formula  of the non-degenerate sequence $v_n :=  \rho u_n $ are all algebraic integers, then the desired result  follows by applying the first case of the Theorem \ref{thm} with $\epsilon / 2$ instead of $\epsilon$. 
 
 Thus, in the rest of the memoir, we will assume $\rho = 1$ meaning that  all the coefficients of the polynomials $P_i$ are algebraic integers.  Put
$$\mathcal{N} := \{(P_1(n)\alpha_1^n, \cdots,  P_m(n)\alpha_m^n), ~~~ n \in \mathbb{N}\}$$ and $X_n = ( P_1(n)\alpha_1^n, \cdots,  P_m(n)\alpha_m^n)$.
 Our main result is a consequence of the following theorem.
  
  \begin{thm} \label{thm1}
For each $0<\epsilon < 1/ 12$,  the set  of  $X_n \in \mathcal{N} $ with $n> \tau(\epsilon)$ solutions of the inequality 
\begin{equation} \label{eq:2}
 \prod_{ v \in S}^{} \prod_{i=1}^{m} \vert l_{iv}(X_n)\vert_v <  \vert \alpha_1  \vert ^{-n\epsilon} 
 \end{equation} 
 is contained in at most  $$2 \cdot  s^2 ( m^{7} 2^{4m+17} \epsilon ^{-2})^{ms+m}$$  proper subspaces in $K^m$.
 \end{thm}

\section{Deduction of Theorem \ref{thm} from  Theorem \ref{thm1}.}
  
Fix $\epsilon >0.$ Let $n$ be a solution of the inequality \eqref{eq:1} with $n > \tau(\epsilon / 2d)$  and  $X_n \in \mathcal{N}$.  For $v \in S$, we have 
\begin{equation} \label{eq:3}
\vert P_i(n)\vert_v \leq ((d_i+1)n^{d_i})^{s(v)}A^{s(v)} \leq  (an^{a})^{s(v)}A^{s(v)},
 \end{equation}  
 where $a$ is the maximum of multiplicity of roots of characteristic polynomial of $\{u_n\}$; $d_i$ the degree of $P_i$ and $A$ is defined in \eqref{eq:b}. By using the relation \eqref{eq:3}, \eqref{eq:d}, the product formula  and the fact that $\alpha_i \in \mathcal{O}_S^*$, we obtain 
 \begin{equation} \label{eq:4}
 \prod_{ v \in S}^{} \prod_{i=1}^{m} \vert P_i(n) \alpha_i^n\vert_v  \leq   (an^{a}A)^{m}.
 \end{equation}
 Hence, 
 $$
  \begin{array}{lcl}
   { \displaystyle \prod_{ v \in S}^{} \prod_{i=1}^{m} \vert l_{iv}(X_n)\vert_v  }   &=& \dfrac{\vert u_n \vert_{{\rm id}}}{ \vert P_1(n)\alpha_1^n \vert_{{\rm id}}}  \cdot  { \displaystyle \prod_{ v \in S}^{} \prod_{i=1}^{m}  \vert P_i(n) \alpha_i^n \vert_v } \\\\
                                                                                                   &\leq& \dfrac{  (an^{a}A)^{l}}{ \vert P_1(n)\alpha_1^n\vert_{id}} \cdot  \vert \alpha_1 \vert^{ns(id)(1-\epsilon)}\\\\
                                                                                                  &\leq& n^{l^2} \cdot \vert \alpha_1 \vert^{-n\epsilon /d}\\\\
                                                                                                   &\leq&   \vert \alpha_1 \vert^{-n\epsilon /2d},\\\\
 
\end{array}
$$
   where for the first inequality, we used  relation \eqref{eq:4},  the fact that $n$ is solution of inequality \eqref{eq:1} and $m\leq l$ and for the last one, \cite[Lem.\,2.4]{AN} and the fact that $n > \tau(\epsilon / 2d)$. So  $X_n$ is solution of inequality  \eqref{eq:2} ( replacing $\epsilon$ by  $\epsilon /2d$).  
Applying Theorem \ref{thm1}  and the fact that  $X_n \in \mathcal{N}$, we deduce that there are at most $$2\cdot s^2 \left(  d^2l^{7} 2^{4l+19} \epsilon ^{-2}\right)^{ls+l}$$ exponential Diophantine equations of the shape 
   \begin{equation} \label{eq:c}
  \sum_{l \in I}^{} c_lP_l(n)\alpha_l^{n}=0
 \end{equation}
 with $c_l's$,  non-zero elements of   $ K$   and $I \subseteq  \{1, 2, \cdots, m\}$ such that  for all  $X_n$  solution of inequality  \eqref{eq:2},  $n$ is solution of at least one of them.   
   From Lemma 2, each Diophantine equation given by the relation \eqref{eq:c} has at most  $ \exp \left( (7l^a)^{8l^a}\right)$ solutions with $n \in \mathbb{Z}$. Hence, we conclude that,  the cardinality of the set of solutions $n$ of the inequality \eqref{eq:1} with $n > \tau(\epsilon / 2d)$  does not exceed  $$ 2\cdot s^2 \left(  d^2l^{7} 2^{4l+19} \epsilon ^{-2}\right)^{ls+l} \exp \left( (7l^a)^{8l^a}\right)$$ which give us the desired result.

\section{preliminaries}\label{sec2}

 Before starting proof of Theorem \ref{thm1}, we need some notations and fact about exterior algebra, inequalities related to heights and results from the geometry of numbers
over number field.  

Let $V$ be a vector space over the splitting field $K$ of the characteristic polynomial P of $\{u_n\}$ of dimension $m$ (number of distinct roots of $P$), with canonical basis $e_i ~~ (i=1, \cdots, m)$. For $1 \leq p \leq m$, we equip the exterior power  $\wedge^p  V$ with the standard basis $$e_{i_1} \wedge e_{i_2}  \wedge \cdots \wedge e_{i_p}, \quad
\mbox{with} \quad i_1<i_2<\cdots < i_p. $$ Put  $ N= \binom{m}{p}$. Let $\sigma_1, \cdots, \sigma_N$ be the subsets of  $\{1, \cdots, m\}$ given by:
$$\sigma_1=\{1, \cdots, p\},~~ \sigma_2=\{1, \cdots, p-1, p+1 \},  ~~\cdots, $$ and $$ \sigma_{N-1}=\{m-p, m-p+2, \cdots, n\},~~\sigma_N=\{m-p+1, \cdots, m\}$$ ordered lexicographically. For the vectors  $x_t=(x_{t1}, \cdots, x_{tm}) \in K^m$, $1 \leq t \leq m,$  we define 

\begin{equation*}
\Delta_{\sigma_i}=\Delta_i(x_1, \cdots, x_p) = 
\begin{vmatrix}
x_{1,i_1} & x_{1,i_2} & \cdots & x_{1,i_p} \\
x_{2,i_1} & x_{2,i_2} & \cdots & x_{2,i_p} \\
\vdots & \vdots  & \ddots &  \vdots\\
x_{p,i_1} & x_{p,i_2} & \cdots & x_{p,i_p} \\
\end{vmatrix}
\end{equation*}

where $\sigma_i=\{i_1, \cdots, i_p\}$ and also $$x_{1} \wedge x_{2}  \wedge \cdots \wedge x_{p} = (\Delta_{\sigma_1},  ~\Delta_{\sigma_2}, ~\cdots, ~ \Delta_{\sigma_N}).$$ Note that $x_{1} \wedge x_{2}  \wedge \cdots \wedge x_{p}  =0$ if and only if $\{x_1, x_2, \cdots,  x_p\}$ is linearly dependent (see for e.g. Lemma 6B in \cite{SB}, Chap.IV). For $x=(x_1,  x_2, \cdots,  x_m)$, we define $x^* = (x_m, - x_{m-1}, \cdots, (-1)^{m-1}x_1)$ and extending the standard scalar product $x\cdot y = {\displaystyle \sum_{i=1}^{m} x_iy_i },$ where $x, y \in K^m$ to $\wedge^p  V$ by the Laplace identity 
 \begin{equation} \label{eq:5}
 (x_{1} \wedge x_{2}  \wedge \cdots \wedge x_{p}) \cdot (y_{1} \wedge y_{2}  \wedge \cdots \wedge y_{p}) = {\rm det}(x_i\cdot y_j)_{i,j=1,\cdots, p}
 \end{equation}
and the Laplace extension 
\begin{equation} \label{eq:6}
x_1\cdot  (x_{2} \wedge x_{3}  \wedge \cdots \wedge x_{m})^* = {\rm det}(x_1\cdots , x_m)  \quad \mbox{with} \quad x_1,\cdots, x_m \in K^m.
 \end{equation}
  A similar notation will be used with regard to  linear forms. If $l(X) = a \cdot X = \sum_{i=1}^{m} a_iX_i, $ where $a=(a_1, \cdots, a_m)$, we put $l^*(X)= a^*\cdot X.$ Further, for linear forms  $l_t(X)= a^{(t)}\cdot X$, with $1 \leq t \leq p,$ in $m$ variables, we define the linear forms in 
 $N$ variables by $$(l_{1} \wedge l_{2}  \wedge \cdots \wedge l_{p})(x)= (l_{1} \wedge l_{2}  \wedge \cdots \wedge l_{p})\cdot x.$$ Let $\{a_1, \cdots, a_m\}$ and  $\{b_1, \cdots, b_m\}$ two basis of $K^m$, then $\{A_1, \cdots, A_N\}$ and  \\
$\{B_1, \cdots, B_N\}$ with $A_i= a_{i_1} \wedge a_{i_2}  \wedge \cdots \wedge a_{i_{m-p}}$ and  $B_i= b_{i_1} \wedge b_{i_2}  \wedge \cdots \wedge b_{i_{m-p}}$  are two basis of $K^N$. \\
    The following lemma, established in (\cite[Lem.\,3.3.1]{EE}, Chap IV) will be used at several places in this  paper.
\begin{lem} \label{lem}
Let $k \in \{1, \cdots, m\}$. If  $\{a_1, \cdots, a_m\}$ and  $\{b_1, \cdots, b_m\}$ two basis of $K^m$, then $\{a_1, \cdots, a_k\}$ and  $\{b_1, \cdots, b_k\}$ generate the same vector space if and only if  $\{A_1, \cdots, A_{N-1}\}$ and  
$\{B_1, \cdots, B_{N-1}\}$  generate the same vector space with  $N=  \binom{m}{k} $.
 \end{lem}

This lemma implies the existence of an injective map $f_k$ from the collection of $k$-dimensional subspaces of $K^m$ to the collection of  $(N-1)$-dimensional subspaces of $K^N$ such that if $V$ has basis  $\{a_1, \cdots, a_k\}$  and $a_{k+1}, \cdots a_n$ are any vectors such that $\{a_1, \cdots, a_n\}$  is a basis of $K^m,$ then $\{A_1, \cdots, A_{N-1}\}$ is a basis of $f_k(V)$.

\begin{lem}\label{lm}
 Let $v \in M_K$. The  following inequalities hold
\begin{equation} \label{eq:7}
\vert x\cdot y \vert_v \leq m^{s(v)} \vert x \vert_v  \cdot \vert y \vert_v   \quad \mbox{for} \quad x, y  \in K^m
 \end{equation}
and 
\begin{equation} \label{eq:9}
\vert x_{1} \wedge x_{2}  \wedge \cdots \wedge x_{k}  \vert_v \leq (k!)^{s(v)} \vert x_1 \vert_v  \cdots  \vert x_k \vert_v   \quad \mbox{for} \quad x_1, \cdots, x_k  \in K^m,~~k \leq m.
 \end{equation}
If  $x_1, \cdots, x_m$ are linearly  independent vectors in $K^m,$ we also have 
\begin{equation} \label{eq:8}
\vert {\rm det}(x_1\cdots , x_m)  \vert_v \leq (m!)^{s(v)} \vert x_1 \vert_v  \cdots  \vert x_m \vert_v 
 \end{equation}
and
\begin{equation} \label{eq:10}
H(x_1 \wedge x_{2}  \wedge \cdots \wedge x_{k} ) \leq k! \cdot H(x_1)\cdot H(x_2) \cdots H( x_k).
 \end{equation}
\end{lem}
\begin{proof}
Let $x, y$ be two elements of $K^m$. By the above definition  of the scalar product on the vector space $K^m$, we infer that
$$\vert x\cdot y \vert_v   = \left| \sum_{i=1}^{m} x_iy_i \right|_v \leq  m^{s(v)} \vert x \vert_v  \cdot \vert y \vert_v,$$
where for the inequality, we have used the fact that  $\vert \cdot \vert_v^{1/s(v)}$ is the usual absolute value when $v$ is an infinite place over $K$ and the strong triangle inequality property of  $\vert \cdot \vert_v$ whereas $v$ is an finite place over $K$. Therefore the relation \eqref{eq:7} holds.  By the  above definition  of the exterior power over the vector space $K^m,$ we have  $$x_{1} \wedge x_{2}  \wedge \cdots \wedge x_{k} = (\Delta_{\sigma_1},  ~\Delta_{\sigma_2}, ~\cdots, ~ \Delta_{\sigma_N}),$$ with $N=\binom{m}{k}$ and
\begin{equation*}
\Delta_{\sigma_i}=\Delta_i(x_1, \cdots, x_k) = 
\begin{vmatrix}
x_{1,i_1} & x_{1,i_2} & \cdots & x_{1,i_k} \\
x_{2,i_1} & x_{2,i_2} & \cdots & x_{2,i_k} \\
\vdots & \vdots  & \ddots &  \vdots\\
x_{k,i_1} & x_{k,i_2} & \cdots & x_{k,i_k} \\
\end{vmatrix}
\end{equation*}
where $\sigma_i=\{i_1< \cdots < i_k\}$. From the Leibniz formula for the determinant $\Delta_{\sigma_i}$,  it follows that 
$$\vert \Delta_{\sigma_i} \vert_v \leq (k!)^{s(v)} \vert x_1 \vert_v  \cdots  \vert x_k \vert_v$$ for all $i \in \{1, \cdots, N\}.$ So  $$\vert x_{1} \wedge x_{2}  \wedge \cdots \wedge x_{k}  \vert_v \leq (k!)^{s(v)} \vert x_1 \vert_v  \cdots  \vert x_k \vert_v.$$ If  $x_1, \cdots, x_m$ are linearly  independent vectors in $K^m,$ then $$\vert {\rm det}(x_1\cdots , x_m) \vert_v =  \vert x_{1} \wedge x_{2}  \wedge \cdots \wedge x_{m}  \vert_v$$ and so the relation \eqref{eq:8} follows immediately from \eqref{eq:9}. By  taking the product over all the places of $M_K$ on the relation \eqref{eq:9} and using the definition of the height function $H(\cdot)$, we obtain the relation \eqref{eq:10}.
\end{proof}

\begin{lem} \label{lem'}
Let $v \in M_K$, $H$ positive real number and $x_1, \cdots, x_p$  linearly  independent vectors in $K^m,$ with $H(x_i) \leq H$ for $i=1, \cdots, p.$ Then
\begin{equation} \label{eq:11}
\frac{H^{-p}}{p!} \leq \dfrac{\vert x_{1} \wedge x_{2}  \wedge \cdots \wedge x_{p}  \vert_v} {\vert x_1 \vert_v  \cdots  \vert x_p\vert_v } \leq (p!)^{s(v)}.
 \end{equation}
 In particular, if $p=m,$ then
 \begin{equation} \label{eq:12}
\frac{H^{-m}}{m!}  \leq \dfrac{\vert {\rm det}(x_1\cdots , x_m) \vert_v } {\vert x_1 \vert_v  \cdots  \vert x_m \vert_v } \leq (m!)^{s(v)}.
 \end{equation}
 \end{lem}
 
\begin{proof}
 We prove this result by using the same argument as in the proof of \cite[Lem.\,2]{E}. Since  $x_1, \cdots, x_p$  are  linearly  independent vectors, it follows from (Lemma 6B \cite{SB}, Chap.IV) that  $x_1 \wedge x_{2}  \wedge \cdots \wedge x_{p} \neq 0$, therefore $H(x_1 \wedge x_{2}  \wedge \cdots \wedge x_{p}) \geq 1$.  Hence we have
  $$
  \begin{array}{lcl}
  \dfrac{\vert x_{1} \wedge x_{2}  \wedge \cdots \wedge x_{p}  \vert_v} {\vert x_1 \vert_v  \cdots  \vert x_p\vert_v }  &=&  \dfrac{ \left({ \displaystyle \prod_{w \in M_K \setminus \{v\}}^{}\vert x_{1} \wedge x_{2}  \wedge \cdots \wedge x_{p}  \vert_v} \right)^{-1} \cdot  H( x_1 \wedge x_{2}  \wedge \cdots \wedge x_{p} ) } {\vert x_1 \vert_v  \cdots  \vert x_p\vert_v }  \\\\
                                                                                                   &\geq& { \displaystyle \prod_{w \in M_K \setminus \{v\} }^{} (p!)^{-s(w)}} \cdot { \displaystyle \prod_{w \in M_K }^{} ( \vert x_1 \vert_w  \cdots  \vert x_p\vert_w)^{-1} } 
                                                                                                   \\\\
                                                                                                  &\geq& (p!)^{-1} H(x_1)^{-1} \cdots H(x_p)^{-1} \\\\
                                                                                                   &\geq&   \dfrac{H^{-p}}{p!}\\
 
\end{array}
$$
where, for the first inequality, we have used  relation \eqref{eq:9} of the Lemma \ref{lm}, the second inequality follows from the fact that  $$\sum_{w \in M_K}^{} s(w) = 1$$ whereas the last inequality holds because $H(x_i) \leq H$ for $i=1, \cdots, p.$  So we proved the first inequality of relation \eqref{eq:11}.  The second inequality follows immediately from relation \eqref{eq:9}. Therefore the proof of the  relation \eqref{eq:11} is complete.  By the definition above of the exterior power of the vector space $K^m,$ ones get 
$$\vert {\rm det}(x_1\cdots , x_m) \vert_v =  \vert x_{1} \wedge x_{2}  \wedge \cdots \wedge x_{m}  \vert_v.$$ Hence, the relation  \eqref{eq:12} follows from the  relation \eqref{eq:11} which complete the proof of our lemma.
\end{proof}
We now define ${\rm rank}(x_1, \cdots, x_r)$ the maximum number of vector in $M = \{x_1, \cdots, x_r\}$ whose are linearly independent over $K^m$. We will need the following lemma which is a slight modification  of Lemma 3 in \cite{E}

\begin{lem} \label{lem1''}
Let $H$ be a positive real number,  $a_1, \cdots, a_r \in \overline{\mathbb{Q}}^m $ with  $H(a_i) \leq H$ for $i=1, \cdots, r$ and $x\in K^m$ non-zero  such that  
$$ a_i \cdot x = 0 \quad \mbox{for} \quad i=\{1, \cdots r\}. $$
Suppose that  ${\rm rank}(a_1, \cdots, a_r) \leq m-1.$ Then, there exists $y \in K^m$ only depend of $a_i$ such that $y \neq 0,~~ a_i \cdot y =0$ for $i=1, \cdots, r$ and $H(y) \leq (m-1)! H^{m-1}$.
 \end{lem}

Here we will provide the result of Bombieri and Vaaler  about a generalization of Minkowski's result on geometry of numbers to Adele ring of number field (see appendix C in \cite{BV} for more details). Throughout of this section,  $K$ is the splitting field of  characteristic polynomial of $\{u_n\}$ and $S$ the set of places consist of all archimedean places and prime ideals above $\alpha_1, \cdots, \alpha_m.$ Let $\mathcal{O}_v$ be a local ring $\{x \in K_v~:~ \vert x \vert_v \leq 1\}.$ A set $C_v$ of $K_v^m$ (with respect to $v$-adic topology) is called a symmetric convex body in $K_v^m$ if 
\begin{itemize}
\item $0$ is an interior point of $C_v$ and $C_v$ compact,
\item If $x \in C_v$ and $\alpha \in K_v, ~~ \vert \alpha \vert_v=1$  then $\alpha x \in C_v,$
\item if $v\mid \infty$ and $ x, y  \in C_v$  then $\lambda x + (1-\lambda) y \in C_v$ for all $\lambda \in \mathbb{R}, ~~0\leq \lambda \leq 1;$ otherwise  $ x, y  \in C_v$  then $x +  y \in C_v.$ 
\end{itemize}
The ring of $K$-Adele $V_K$ is the set of infinite tuple $(x_v, ~~ v\in M_K)$ with $x_v \in K_v$ and  $\vert x_v \vert_v \leq 1$ for all but finitely many $v$, endowed with component-wise addition and multiplication.  The set $V_K^m$ may be identified with a set of infinite tuples of vectors $(x_v)$ with $x_v \in K_v^m$ for all $v \in M_K$ and $x_v \in \mathcal{O}_v^m$ for all but finitely many $v$. By the diagonal embedding, we mean the map $$\psi\,:\, K^m \rightarrow V_K^m \,:\, x \mapsto (x_v),$$ with $x_v = x$ for $v \in M_K.$ A convex body in  $V_K^m$ is a Cartesian product  $$C = \prod_{v \in M_K} C_v = \{(x_v) \in V_K^m ~:~~ x_v \in C_v ~~ \mbox{for} ~~ \mbox{all} ~~ v \in M_K\}$$ where for every $v \in M_K, ~~ C_v$ is a symmetric convex body in  $V_K$ and 
$C_v = \mathcal{O}_v^m$ for all but finitely many $v \in M_K.$ For $\lambda \in \mathbb{R},$ we define $$\lambda C := \prod_{v \mid \infty}^{} \lambda C_v \times  \prod_{v \nmid \infty}^{}  C_v .$$ The $i^{th}$ successive minima $\lambda_i := \lambda_i (C)$ is defined by $$\lambda_i := \min \{\lambda \in \mathbb{R}_{>0}~:~~ \psi^{-1} (\lambda C) ~~ \mbox{contains} ~ i~ \mbox{linear}~ \mbox{independent}~\mbox{points}\}.$$ This minimum exist (see Appendix C in \cite{BV}). We recall that we use the volume with respect to the following measure. If $v \nmid \infty,$ then $K_v$ is a commutative locally compact set so that the Haar measure exist and it's unique up to a constant. Further, we consider the measure $\beta_v$ on $K_v$ defined by $\beta_v (\Omega) = \vert \Delta_{K_v } \vert_p^{1/2} \mu_v (\Omega), $ with $\mu_v$ defined at Example C.1.3 \cite{BV} and $\Omega$ a $\mu_v$-measurable set of $K_v$. If $v \mid \infty$  and real then $K_v = \mathbb{R}$ and take $\beta_v$ the usual Lebesgue measure on $\mathbb{R.}$ If $v$ is complex infinite, then  $K_v = \mathbb{C} $ and take $\beta_v$ two times Lebesgue measure on the complex plane.  The corresponding product measure  on $K_v^m$ is denoted $\beta_v^m.$ If $\rho$ is a linear transformation of $K_v^m$ into itself then
 $$\beta_v^m (\rho D) = \vert {\rm det}(\rho) \vert_v^d \beta_v^m(D)$$ for any $\beta_v^m$-measurable set $D$ of $K_v^m$ (see pp. 239 in \cite{E}). Now let $\beta= \prod_{v}\beta_v$ be the product measure on $V_K$ and $\beta^m$ the $n$-fold product measure of $V_K^m.$\\
      {Bombieri and Vaaler (\cite{BW}, Theorem 3  and  Theorem 6) proved the following generalisation of Minkowski's theorem (see also Lemma 6 in \cite{E}).
  
 \begin{lem} \label{lem1}
Let $\Delta_K$ be the discriminant of $K$ and $r_2$ the number of complex infinite places of $K$. Further let $C$ be a symmetric convex body in $V_K^m$ and $\lambda_1, \cdots, \lambda_m$ its successive minima. Then 
$$\left( \dfrac{\pi^mm!}{2}\right)^{r_2/d} \dfrac{2^m}{m!} \vert \Delta_K \vert^{-m/2d} ~ \leq ~  \lambda_1 \cdots  \lambda_m \cdot \beta^m(C)^{1/d} ~ \leq  ~ 2^m.$$
 \end{lem}
 By a tuple $(N, \overline{\gamma}, \mathcal{L}, Q)$ 
\begin{itemize}
\item $Q>1;~~N \geq 2; \quad \overline{\gamma}= (\gamma_{iv}, ~~ v \in S, ~~ i=1, \cdots, N);$
\item a system of linear forms $\mathcal{L} = \{\overline{l}_{iv}~:~ v \in S, ~~ i=1, \cdots, N\}$ in $N$ variables such that  $ \{\overline{l}_{1v}, \cdots,  \overline{l}_{Nv} \} $ is linearly independent for each $v \in S,$
\end{itemize}
we define the set $$\Pi (N, \overline{\gamma}, \mathcal{L}, Q) := \{ y \in \mathcal{O}_S^N~:~ \vert \overline{l}_{iv}(y) \vert_v \leq Q^{\gamma_{iv}}, ~~ \mbox{for} ~~ v \in S, ~~ i=1, \cdots, N\}$$ and 
$$V(N, \overline{\gamma}, \mathcal{L}, Q) := the ~ K-vector ~ space ~ generated ~ by ~ \Pi (N, \overline{\gamma}, \mathcal{L}, Q).$$ Notice that the map $f_k$ is defined at the page 7. We will deduce our Theorem \ref{thm1} from the two following results.

\begin{thm} \label{thm2}
Let   $0 <  \epsilon < 1/ 12,$ and $X_n \in  \mathcal{N}$ with $n> \tau(\epsilon)$ solution of inequality \eqref{eq:2} such that $l_{iv}(X_n) \neq 0$  for $v \in S,~~ i=1, \cdots, m.$  There are a proper linear subspace $W$ of $K^m$ and a tuple $ (N, \overline{\gamma}, \mathcal{L}, Q)$ with $N= \binom{m}{k}$ where $k = {\rm dim}_K W,$ such that  the three conditions below are satisfied:      
 \begin{equation} \label{eq:13}
X_n \in W, \quad f_k(W) = V(N, \overline{\gamma}, \mathcal{L}, Q), \quad {\rm dim}_K (V(N, \overline{\gamma}, \mathcal{L}, Q))= N-1;
 \end{equation}

 \begin{equation} \label{eq:14'}
 \overline{\gamma}= (\gamma_{iv}, ~~ v \in S, ~~ i=1, \cdots, N), ~~ with ~~  \gamma_{iv} \leq s(v) ~~and ~~ \sum_{v \in S}^{} \sum_{i=1}^{N} \gamma_{iv} \leq -\frac{\epsilon}{9m^3};
 \end{equation}

 \begin{equation} \label{eq:15"}
Q \geq (\vert \alpha_1 \vert)^{3mn}
 \end{equation}
and such that $(N, \overline{\gamma}, \mathcal{L}) \in \mathcal{C},$ where $\mathcal{C}$ is a fixed set independent of $X_n$ of cardinality at most   $C':= (64em^42^m \epsilon^{-1})^{ms+m}.$
\end{thm}

The next result gives us an effective upper bound of the number of vector spaces $V(N, \overline{\gamma}, \mathcal{L}, Q)$ when $Q$ is large enough.

\begin{thm} \label{thm3}
Fix a tuple $(N, \overline{\gamma}, \mathcal{L}) $ as in Theorem \ref{thm2}.  Then  there exist a collection of $(N-1)$-dimensional linear subspaces of $K^N$ of cardinality at most  $$C" =  2^{11m+34}s^2 \epsilon^{-4} m^{12}$$ such that for every $Q$ with 
 \begin{equation} \label{eq:13'}
{\rm dim}_K (V(N, \overline{\gamma}, \mathcal{L}, Q))= N-1; \quad Q> ( 9m!)^{e^{2C"}},
 \end{equation}
the vector spaces $V(N, \overline{\gamma}, \mathcal{L}, Q)$ belongs to this collection.
 \end{thm}
 
By combining Lemma \ref{lem}, Theorem \ref{thm2} and Theorem \ref{thm3}, we easily deduce Theorem \ref{thm1}. Indeed, we notice that  lemma \ref{lem} implies  the map $f_k$ from the collection of $k$-dimensional subspaces of $K^m$ to the collection of  $(N-1)$-dimensional subspaces of $K^N$ such that, for any $k$-dimensional subspace $V$ of  basis  $\{a_1, \cdots, a_k\}$  and $a_{k+1}, \cdots a_n$ are any vectors such that $\{a_1, \cdots, a_n\}$  is a basis of $K^m,$ then   $f_k(V)$ is the subspace generated by $\{A_1, \cdots, A_{N-1}\}$, with $N =\binom{m}{m-k}$ with $A_i= a_{i_1} \wedge a_{i_2}  \wedge \cdots \wedge a_{i_{m-k}}$ is injective.  Let   $0 <  \epsilon < 1/ 12$ and $X_n \in  \mathcal{N}$ with $n> \tau(\epsilon)$ solution of inequality \eqref{eq:2} such that $l_{iv}(X_n) \neq 0$ for $v \in S,~~ i=1, \cdots, m.$ Since  $n> \tau(\epsilon)$, then we deduce $$\vert \alpha_1 \vert ^{3mn} \geq ( 9m!)^{e^{2C"}} \quad \mbox{with} \quad C" = 2^{11m+34}s^2 \epsilon^{-4} m^{12}.$$ It follows from Theorem \ref{thm2} the existence of a proper subspace $W$ of $K^m$ and a  tuple $(N, \overline{\gamma}, \mathcal{L})$  independent of $n$   such that $$X_n \in W, \quad f_k(W) = V(N, \overline{\gamma}, \mathcal{L}, Q), \quad {\rm dim}_K (V(N, \overline{\gamma}, \mathcal{L}, Q))= N-1, \quad Q \geq  ( 9m!)^{e^{2C"}}.$$  The cardinality of the collection of such tuples $(N, \overline{\gamma}, \mathcal{L})$ does not exceed $C':= (64em^42^m \epsilon^{-1})^{ms+m}.$ For each tuple $(N, \overline{\gamma}, \mathcal{L})$, we deduce from Theorem \ref{thm3}, that the collection of vector spaces $V(N, \overline{\gamma}, \mathcal{L}, Q)$ satisfying the properties $$ {\rm dim}_K (V(N, \overline{\gamma}, \mathcal{L}, Q))= N-1 \quad Q\geq  ( 9m!)^{e^{2C"}}$$ does not exceed $C"$. Since $f_k$ is injective, we obtain the set of proper subspace $W$ of $K^m$ for which $X_n \in W$ does not exceed $$C'C" =  (64em^42^m \epsilon^{-1})^{ms+m} \cdot 2^{11m+34}s^2 \epsilon^{-4} m^{12} \leq   s^2 ( m^{7} 2^{4m+17} \epsilon ^{-2})^{ms+m}.$$ By adding the collection of $ms$ proper subspaces 
associated to the collection of $\{l_{iv}, ~ v\in S,~ i=1,\cdots, m\},$  we conclude that the set of $X_n \in  \mathcal{N}$ with $n> \tau(\epsilon)$ solution of inequality \eqref{eq:2} is contained in at most $$ms + s^2 ( m^{7} 2^{4m+17} \epsilon ^{-2})^{ms+m} \leq 2\cdot s^2 ( m^{7} 2^{4m+17} \epsilon ^{-2})^{ms+m}$$ proper subspaces in $K^m.$

\section{Proof of Theorem \ref{thm2}}\label{sec3}

We prove this result using the same machinery as in the proof of Theorem B in \cite{E}, with the difference that the norm vector in $K^m$ that we are working with it's not a Euclidean norm that  Evertse  consider in his paper. Let   $0 <  \epsilon < 1/ 12$ and $X_n \in  \mathcal{N}$ with $n> \tau(\epsilon)$ solution of inequality \eqref{eq:2} such that $l_{iv}(X_n) \neq 0$ for $v \in S,~~ i=1, \cdots, m.$ Fix  $v \in S$, we define   $\mathcal{O}_v =\{x \in K_v~:~ \vert x \vert_v \leq 1\}$ and the $v$-adic approximation domain $\Pi_v(n)$ of $X_n$ to be the paralellepiped in $K_v^m$ given by  $$\Pi_v(n)= \{ y \in K_v^m ~:~ \vert l_{iv}(y) \vert_v \leq \vert l_{iv}(X_n) \vert_v ~~ \mbox{for} ~~  i=1, \cdots, m\}.$$ This is a compact convex body of $K_v^m$. The approximate domain is defined in $V_K^m$ by 

 $$
  \begin{array}{lcl}
   \Pi (n)   &=& \prod_{v \in S}^{}  \Pi_v(n) \times  \prod_{v \notin S}^{}  \mathcal{O}_v^m\\\\
                &=& \{ y_v \in V_K^m ~:~ \vert l_{iv}(y_v) \vert_v \leq \vert l_{iv}(X_n) \vert_v ~~ \mbox{for} ~~v \in S,
                 ~~\mbox{and} ~~\vert y_v \vert_v \leq 1 ~~ \mbox{for} ~~v \notin S \}.
\end{array}
$$
We denote by $\lambda_1(n), \cdots, \lambda_m(n)$ the successive minima of $\Pi(n)$; $$ R_v(n) =  \left(\prod_{i=1}^{m} \vert l_{iv}(X_n)\vert_v \right)^{-1} ~~ \mbox{and} ~~  R(n)= \prod_{v \in S}^{} R_v(n).$$

\begin{lem} \label{lem3}
Let  $v\in M_K$ and $D_{K_v}$ the discriminant over $K_v$. When $v$ is an finite place over $K$, we denote by $ p$ the rational prime above $v$.  
\begin{enumerate}
\item If $v\notin S,$ then $\beta_v^m(n) := \beta_v^m(\Pi_v(n)) = \beta_v^m(\mathcal{O}_v^m) = \vert D_{K_v} \vert_p^{m/2}$.\\

\item  If $v\in S$  and not archimedean,  then $\beta_v^m(n)  =  \vert D_{K_v} \vert_p^{m/2} \cdot R_v(n)^{-d}$.\\

\item  If $K_v = \mathbb{R}$,  then $\beta_v^m(n)  =  2^m \cdot R_v(n)^{-d}$.\\

\item If $K_v = \mathbb{C}$,  then $\beta_v^m(n)  =  (2\pi)^m \cdot R_v(n)^{-d}$,\\

\item  and we also  have   $$\dfrac{1}{m!} R(n)  \leq \lambda_1(n)  \cdots  \lambda_m(n)  \leq \Delta_K^{m/2d} R(n).$$
\end{enumerate}
 \end{lem}

\begin{proof}
It is easy to see that, for $v \in S,$ we have $$\Pi_v(n)= \rho \mathcal{O}_v^m \quad \mbox{with}\quad \rho = \begin{pmatrix}
l_{1v}(X_n) & l_{2v}(X_n)\cdot \{0,-1\} & \cdots & l_{mv}(X_n) \cdot \{0,-1\} \\
0 & l_{2v}(X_n) & \cdots & 0 \\
\vdots & \vdots  & \ddots &  \vdots\\
0 & 0 & \cdots & l_{mv}(X_n) \\
\end{pmatrix},$$ where $l_{iv}(X_n)\cdot \{0,-1\} = -l_{iv}(X_n)$ if $v = {\rm id}$ and $0$ otherwise. 
Hence, the first four statements  follow from the definition of $\beta_v$ and the fact that $$\beta_v^m(n)= \beta_v^m (\rho \mathcal{O}_v^m) = \vert {\rm det}(\rho) \vert_v^d \beta_v^m(\mathcal{O}_v^m).$$ 
  From B.1.20  and  B.1.21  (page 606 in \cite{BV}), we have $$\vert \Delta_K \vert_p = \prod_{v\mid p}^{} \vert \Delta_{K_v} \vert_p. $$ Together with the product formula give us
 \begin{equation} \label{eq:13''}
\vert \Delta_K \vert^{-1} = \prod_{p}\prod_{v\mid p}^{} \vert \Delta_{K_v} \vert_p.
 \end{equation}
By  the definition of $\beta^m,$ we obtain
 $$
  \begin{array}{lcl}
   \beta^m(n)    &=&2^{mr_1} \cdot (2\pi)^{mr_2} {\displaystyle \prod_{p}\prod_{v\mid p}^{} \vert \Delta_{K_v} \vert_p^{m/2}} \cdot (\prod_{v \in S}^{} R_v(n))^{-d}\\\\
   
                &=& 2^{mr_1} \cdot (2\pi)^{mr_2}  \vert \Delta_{K} \vert^{-m/2} \cdot R(n)^{-d},\\\\
\end{array}
$$
where, for the last equality, we used relation \eqref{eq:13''} and $r_1,~r_2$ correspond respectively to the number of real and complex places.
Together with  Lemma \ref{lem1}, we deduce
$$
  \begin{array}{lcl}
  \lambda_1(n)  \cdots  \lambda_m(n)    &\leq &2^{m}  \beta^m(n)^{-1/d} \\\\
   
                &\leq& 2^{m} \cdot 2^{-mr_1/d} (2\pi)^{-mr_2/d}  \vert \Delta_{K} \vert^{m/2d} \cdot R(n),\\\\
                
                &\leq& \left(\dfrac{2}{\pi}\right)^{mr_2/d} \vert \Delta_{K} \vert^{m/2d} \cdot R(n),\\\\
                
                 &\leq&  \vert \Delta_{K} \vert^{m/2d} \cdot R(n),\\\\

\end{array}
$$
and 

$$
  \begin{array}{lcl}
  \lambda_1(n)  \cdots  \lambda_m(n)    &\geq & \left(\dfrac{\pi^mm!}{2}\right)^{r_2/d} \dfrac{2^m}{m!} \cdot 2^{-mr_1/d} (2\pi)^{-mr_2/d}  \cdot R(n)\\\\
   
                &\geq& \left(\dfrac{m!}{2}\right)^{r_2/d} \dfrac{(2^m)^{r_2/d}}{m!} \cdot R(n),\\\\
                
                &\geq& \dfrac{R(n)}{m!}. \\\\           
 
 \end{array}
$$
Hence the proof is completed.
\end{proof}
We also need the following lemma.
\begin{lem} \label{lem2'}
Let $\lambda >0.$ For each $y \in \psi^{-1} (\lambda \Pi(n)),$ we have

\begin{itemize}
\item  $ \vert y \vert_v \leq  \left( m^2(m-1)! R(n)^{\frac{11}{10\epsilon}}\lambda \right)^{s(v)} $  for  $v \in S.$\\
\item  $\lambda_1(n) \geq m^{-2}(m-1)!^{-1} R(n)^{\frac{-11}{10\epsilon}}.  $\\
\end{itemize}
\end{lem}
\begin{proof}
Let  $y \in \psi^{-1} (\lambda \Pi(n))$. Fix $v \in S$ and $\Delta_v = {\rm det} (l_{1v}, \cdots, l_{mv})$. Clearly $\Delta_v =1.$ Let $a_j$ be the  vector $(l_{1v} \wedge \cdots \wedge  l_{j-1v}  \wedge l_{j+1v} \wedge \cdots \wedge l_{mv})^*$. The fact that $\{a_1, \cdots, a_m\}$ is a basis of $K^m$ implies
 \begin{equation} \label{eq:14}
y = \sum_{j=1}^{m} \Delta_v^{-1} l_{jv}(y) a_j  = \sum_{j=1}^{m} l_{jv}(y) a_j.
 \end{equation}
 
 Using the fact $n> \tau(\epsilon),$ one gets $\vert X_n \vert_v \leq (\vert \alpha_1 \vert )^{11ns(v)/10}.$ Namely,  if $v$ is a finite place over $K$, then $\vert X_n \vert_v \leq 1$ as all the coordinates $P_i(n)\alpha_i^n$ of $X_n$ are algebraic integers so we are done.  Otherwise, by using the fact that $\vert \cdot \vert_v^{1/s(v)}$ is the usual absolute value, we get
 $$
\begin{array}{lcl}
\vert X_n\vert_v    &= &  \vert (P_i(n)\alpha_1^n, \cdots, P_m(n)\alpha_m^n) \vert_v \\\\
   
                                    &\leq& \max_i   \vert (P_i(n)\alpha_i^n \vert_v \\\\
                
                                   &\leq& \left(\dfrac{(a+1)A n^a}{\vert \alpha_1 \vert ^{n/10}}\right)^{s(v)} \cdot   (\vert \alpha_1 \vert )^{11ns(v)/10} \\\\         
                                   
                                    &\leq&  (\vert \alpha_1 \vert )^{11ns(v)/10},
 \end{array} 
 $$
  where, for the second inequality, we have used the triangle inequality  of the absolute value and the definition of $A$ in \eqref{eq:b} while for the last inequality, the fact that $n> \tau(\epsilon).$ Therefore
  \begin{equation} \label{eq:15'}
\begin{array}{lcl}
\vert l_{jv} (y) \vert_v    &\leq & \lambda^{s(v)} \vert  l_{jv}(X_n) \vert_v \\\\
   
                                    &\leq&  \left( \lambda m (\vert \alpha_1 \vert)^{11n/10} \right)^{s(v)}\\\\
                
                                   &\leq& \left( \lambda m (R(n))^{11 / (10\epsilon)} \right)^{s(v)}.

 \end{array}
 \end{equation}

 By relation \eqref{eq:9},  we have $$\vert a_j \vert_v = \vert l_{1v} \wedge \cdots \wedge  l_{j-1v}  \wedge l_{j+1v} \wedge \cdots \wedge l_{mv} \vert_v \leq (m-1)!^{s(v)}.$$ Together with  \eqref{eq:14} and \eqref{eq:15'}, we obtain
 $$
 \begin{array}{lcl}
\vert y \vert_v    &\leq & m^{s(v)} \max_j \{\vert l_{jv} (y) a_j \vert_v\}\\
   
                                    &\leq& \left( \lambda m^2 (m-1)! (R(n))^{11 / (10\epsilon)} \right)^{s(v)}. \\
                
 \end{array}
 $$
 Hence the proof of the first statement  is completed.  Let  $y \in \psi^{-1} (\lambda_1(n) \Pi(n))$ with $y \neq 0.$ Then  $\vert y \vert_v   \leq 1$ for $v \notin S$ and since $H(y) \geq 1$, it follows that 
 $$1 \leq H(y) \leq \prod_{v \in S}^{}  \vert y \vert_v \leq m^2 (m-1)! \lambda_1(n)(R(n))^{11/(10 \epsilon)}$$ which implies the second statement of the lemma.
\end{proof}

The following lemma is a particular case of Lemma 15 of Evertse \cite{E}.

\begin{lem} \label{lem4}
Let $ b_1, \cdots, b_m$ be linearly independent vectors in $K^m$. For each infinite place $v$ on $K$, let  $ \mu_{1v}, \cdots, \mu_{mv}$ be reals numbers with $$0< \mu_{1v} \leq \mu_{2v} \leq \cdots \leq \mu_{mv}.$$
Suppose $\vert l_{iv} (b_j)\vert_v \leq \mu_{jv}$ for $1 \leq i,~j \leq m, ~~ v \mid \infty.$ Then there are permutations $k_v$ of $\{1, \cdots, m\}$ for each infinite place $v$ on $K$ and vectors 
$$v_1=b_1, ~~ v_i =  \sum_{j=1}^{i-1} \epsilon_{ij} v_j + b_i  \quad \mbox{with} \quad  \epsilon_{ij} \in \mathcal{O}_K$$ such that 
$$\vert l_{k_v(i), v} (v_j)\vert_v \leq   (2d\vert \Delta_K\vert^{1/2})^{s(v) (i+j)} \min \{\mu_{iv}, \mu_{jv}\},  \quad \mbox{for} \quad     1 \leq i,~j \leq m, ~~ v \mid \infty .$$
 \end{lem}

 Choose linear independent vectors $g_1, \cdots, g_m$ with $g_j \in \psi^{-1} (\lambda_j \Pi(n))$ for any $j=1, \cdots, m.$ By Lemma \ref{lem2'}, we have  $g_j \in \mathcal{O}_K$ and 
 
\begin{equation*}
\left\{
\begin{aligned}
\vert l_{iv} (g_j)\vert_v &\leq \lambda_j^{s(v)}\vert l_{iv} (X_n)\vert_v& \mbox{for} \quad  v \mid \infty,\\\\
\vert l_{iv} (g_j)\vert_v &\leq \vert l_{iv} (X_n)\vert_v& \mbox{for} \quad v\in S \quad v \nmid \infty.\\\\
\end{aligned}
\right.
\end{equation*}

By applying Lemma \ref{lem4} with the vectors $g_1, \cdots, g_m$  and the linearly forms $l_{iv}(X_n)^{-1}l_{iv}$ ($v \mid \infty, ~~ i=1, \cdots, m$);  the real  numbers $\mu_{jv} = \lambda_j^{s(v)}~~ (v \mid \infty, ~~j=1, \cdots, m) $, we infer that there are vectors
$ v_1, \cdots, v_m$ with  
 \begin{equation} \label{eq:15}
v_1=g_1, ~~ v_i =  \sum_{j=1}^{i-1} \epsilon_{ij} v_j + g_i  \quad \mbox{with} \quad  \epsilon_{ij} \in \mathcal{O}_K
 \end{equation}
and permutations $k_v$ of $\{1, \cdots, m\}$ for $v \mid \infty$  such that 
 
 \begin{equation} \label{eq:16}
\vert l_{k_v(i), v} (v_j)\vert_v \leq  \left( (2d\vert \Delta_K\vert^{1/2})^{2m} \lambda_{\min \{i, j\}}\right)^{s(v)}\vert l_{k_v(i),v}(X_n) \vert_v,  \quad \mbox{for} \quad     1 \leq i,~j \leq m, ~~ v \mid \infty. 
 \end{equation} 
 
 Fix $v \in S$, $v \nmid \infty$. Since  $\epsilon_{ij} \in \mathcal{O}_K$ and $g_{j} \in \mathcal{O}_K^m,$ we deduce $$v_j \in \mathcal{O}_K^m, \quad \vert l_{i v} (v_j)\vert_v \leq  \vert l_{iv}(X_n) \vert_v.$$ Hence we have proven the lemma below
 
\begin{lem} \label{lema}
There are linearly independent vectors $v_1, \cdots, v_m \in \mathcal{O}_K^m$ and permutations $k_v$ of $\{1, \cdots, m\}$ for $v \mid \infty$  all depending of $X_n$, with the followings properties: 
\begin{itemize}
\item  for $j=1, \cdots, m$, the vectors $v_1, \cdots, v_j$ belong in the $K$-vector space generated by $ \psi^{-1} (\lambda_j(n) \Pi(n)).$   \\
\item  We also have 
\begin{equation*}
\left\{
\begin{aligned}
\vert l_{k_v(i),v} (v_j)\vert_v &\leq  \left( (2d\vert \Delta_K\vert^{1/2})^{2m} \lambda_{\min \{i, j\}}(n)\right)^{s(v)}\vert l_{k_v(i),v}(X_n) \vert_v, \quad \mbox{for} \quad v\mid \infty& \\\\
\vert l_{iv} (v_j)\vert_v &\leq \vert l_{iv}(X_n) \vert_v \quad  \mbox{for} \quad v\in S, \quad v \nmid \infty.&
\end{aligned}
\right.
\end{equation*}
\end{itemize}
\end{lem}
In the next lemma, we make a first step toward the proof of Theorem \ref{thm1}.

\begin{lem} \label{lem5}
There is a set $\Gamma$ of cardinality $$\vert \Gamma \vert \leq (m!)^{-s} (64em^42^m \epsilon^{-1})^{ms+m}$$  consisting of the tuples of real numbers  \\
 $$(c;d) = (c_{iv} ~:~~v \in S, ~~ i=1, \cdots, m; ~~ d_i~:~~i=1, \cdots, m)$$ 
with 
\begin{itemize}
\item $c_{iv} \leq \dfrac{12s(v)}{10\epsilon}$   for $v \in S, ~~i=1, \cdots, m.$
\item  $d_1 \leq 0, ~~ \dfrac{-12}{10\epsilon} \leq d_1 \leq \cdots \leq d_m$ 
\end{itemize}
such that for every $X_n$ solution of inequality \eqref{eq:2}, with $n > \tau(\epsilon)$ and  $l_{iv}(X_n) \neq 0$ for $v \in S,~~ i=1, \cdots, m$; there is a tuple \\
$(c;d) \in \Gamma$ with $$R(n)^{c_{iv} - \{1/s c(m)\}} < \vert l_{iv}(X_n) \vert_v  \leq  R(n)^{c_{iv} } $$ and 
$$R(n)^{d_{i} }  \leq  \lambda_i(n)   \leq  R(n)^{d_{i}+\{1/c(m)\}} $$ for $v \in S, ~~ i=1, \cdots, m$ with $c(m) = 4m^32^m.$
\end{lem}

\begin{proof}
Put $u_{iv} :=  \vert l_{iv}(X_n) \vert_v  \left( R(n)^{\frac{-12}{10\epsilon}}\right)^{s(v)}$ for $v \in S, ~~i=1, \cdots, m.$  Since  \\
$\vert X_n \vert_v \leq (\vert \alpha_1 \vert  )^{11ns(v)/10}$, it follows that
$$u_{iv}  \leq \vert X_n \vert_v \left(m R(n)^{\frac{-12}{10\epsilon}}\right)^{s(v)} \leq   \left(m R(n)^{\frac{-1}{10\epsilon}}\right)^{s(v)} \leq 1,$$ where for the second inequality, we used the fact that $X_n$ solution of inequality \eqref{eq:2} and for the last inequality,  $n>\tau(\epsilon)$. However
$$
\begin{array}{lcl}
\prod_{v \in S}^{} \prod_{i=1}^{m} u_{iv}  &=&   \left( \prod_{v \in S}^{} \prod_{i=1}^{m}  \vert  l_{iv}(X_n) \vert_v \right)    R(n)^{\frac{-12m}{10\epsilon}} \\\\
   
                                                              &\geq&  \left( R(n) \right)^{\frac{-2m}{\epsilon}}.\\         
 
 \end{array}
 $$
  By applying Lemma 9 in \cite{E} with $\theta = 0.5\epsilon c(m)^{-1}$, there is a set  $\Gamma_1$ of non-negative reals $\beta = (\beta_{iv}, ~~v \in S, ~~ i=1, \cdots, m)$ of cardinality $$\vert \Gamma_1 \vert \leq
  (2e +2ec(m)\epsilon^{-1})^{ms}$$  such that for some $\beta \in \Gamma_1,$ we have 
  \begin{equation} \label{eq:17}
R(n)^{c_{iv} - \{1/s c(m)\}} < \vert l_{iv}(X_n) \vert_v  \leq  R(n)^{c_{iv}}, 
 \end{equation}
with $$c_{iv} = \dfrac{12}{10 \epsilon} s(v) - \dfrac{2m}{\epsilon} \beta_{iv},  \quad \mbox{for} \quad  v \in S, ~~ i=1, \cdots, m. $$  Therefore $c_{iv} \leq  \dfrac{12}{10\epsilon} s(v)$. Finally $c= (c_{iv},~~v \in S,~~ i=1, \cdots, m)$ depend only of 
$\gamma \in \Gamma_1,$ hence for $c$ we have at most  $\vert \Gamma_1 \vert$ possibilities. \\

 We define the numbers  $v_{j}(n) :=   R(n)^{\frac{-12}{10\epsilon}} \lambda_j(n)^{-1}$ for $v \in S, ~~i=1, \cdots, m.$ By  Lemma \ref{lem2'}, we have $\lambda_1(n) \geq m^{-2}(m-1)!^{-1} R(n)^{\frac{-11}{10\epsilon}}$. Together with $n> \tau(\epsilon)$
   we obtain $$1 \geq v_1(n) \geq v_2(n) \geq \cdots \geq v_m(n).$$ Then, it follows from  Lemma \ref{lem3} that
 $$v_1(n)\cdots v_m(n) = \dfrac{R(n)^{\frac{-12m}{10\epsilon}}}{\lambda_1(n)\cdots \lambda_m(n)} \geq R(n)^{\frac{-2m}{\epsilon}}.$$ By applying  Lemma 9 in \cite{E} with $\theta = 0.5\epsilon c(m)^{-1}$,  there is a set  $\Gamma_2$ of $m$-tuples of  reals non-negative
  of $\delta = (\delta_{i},  ~~ i=1, \cdots, m)$  independent of $X_n$ of cardinality $$\vert \Gamma_2 \vert \leq (2e +2ec(m)\epsilon^{-1})^{m}$$  such that for some $\delta \in \Gamma_2,$ 
  we have 
  \begin{equation} \label{eq:18}
R(n)^{\frac{-2m}{\epsilon} \left(\delta_{j} + \frac{\epsilon}{2m c(m)}\right)} < v_j(n)  \leq  R(n)^{-2m \delta_j / \epsilon}, 
 \end{equation}
the relation \eqref{eq:18} remains valid after replacing $\delta_1, \cdots, \delta_m$ by $$\delta_1' = \min\{\delta_1, \cdots, \delta_m\}, \quad \delta_2' = \min\{\delta_2, \cdots, \delta_m\}, \cdots, \delta_m' =  \delta_m$$ respectively. Putting 
$$d_j = \dfrac{-12}{10\epsilon} +  \dfrac{2m \delta_j'}{\epsilon}, \quad \mbox{for} \quad j=1, \cdots, m.$$ Finally, we get  $$R(n)^{d_{i} }  \leq  \lambda_i(n)   \leq  R(n)^{d_{i}+\{1/c(m)\}} $$  therefore $d=(d_1, \cdots, d_m)$
depend only  of $\Gamma_2$. From the definition $d_i$ we have the desired inequality.  Since $X_n \in \psi^{-1} (\Pi(n))$, it follows by definition of $1^{st}$ successive minima, we have $\lambda_1(n) \leq 1$ therefore $d_1 \leq 0$. For such a tuple $d$, we have at most  $\vert \Gamma_2 \vert$ possibilities.
It follows that the numbers of possibilities for $(c;d)$ is at most 
 $$
\begin{array}{lcl}
\vert \Gamma_1 \vert   \vert \Gamma_2\vert &\leq &  (2e +2ec(m)\epsilon^{-1})^{ms+m}\\\\
   
                                                                     &\leq&  (16e^2m^32^m\epsilon^{-1})^{ms+m}\\\\
                                                                     
                                                                      &\leq& (m!)^{-s} (64em^42^m\epsilon^{-1})^{ms+m},\\\\
\end{array}
 $$
 where, we used Stirling's approximation to get the last inequality.
 \end{proof}
Let $X_n$ be a solution of inequality \eqref{eq:2}  with $n> \tau(\epsilon)$ and $(c;d)$ the corresponding tuple from Lemma \ref{lem5}. For any finite place $v \in S$ choose $k_v$ to be identity map from $\{1, \cdots, m\}$ to $\{1, \cdots, m\}$. We define $$l_{iv}'(X) =  l_{k_v(i),v}(X)  \quad \mbox{for} \quad v \in S,~~~ i=1, \cdots, m$$ and the numbers  $$e_{iv} =  c_{k_v(i),v}  \quad \mbox{for} \quad v \in S,~~~ i=1, \cdots, m.$$ We have constructed a tuple 
$$\mathcal{T} = (l_{iv}'~:~~ v \in S, ~~ i=1, \cdots, m; ~~~ e_{iv}~:~~ v\in S, ~~ i=1, \cdots, m; ~~~ d_i~:~~   i=1, \cdots, m).$$  By Lemma  \ref{lem5}, we have 
\begin{equation} \label{eq:19}
R(n)^{e_{iv} - \{1/s c(m)\}} < \vert l_{iv}'(X_n) \vert_v  \leq  R(n)^{e_{iv} } 
 \end{equation}
 and 
 \begin{equation}\label{eq:20}
R(n)^{d_{i} }  \leq  \lambda_i(n)   \leq  R(n)^{d_{i}+\{1/c(m)\}}
 \end{equation}
 for $v \in S, ~~ i=1, \cdots, m.$

  \begin{lem} \label{lem6'}
$\mathcal{T}$ belongs to a set independent of $X_n$ of cardinality at most  $$C' = (64em^42^m \epsilon^{-1})^{ms+m} $$
we also have 
\begin{itemize}
\item $e_{iv} \leq \dfrac{12s(v)}{10\epsilon}$   for $v \in S, ~~i=1, \cdots, m.$\\
\item  ${\displaystyle \sum_{v \in S}^{}  \sum_{i=1}^{m}  e_{iv}}  \leq  -1 + (4m^22^m)^{-1} $\\
\item  $1 - (3m^22^m)^{-1} \leq {\displaystyle \sum_{i=1}^{m}  d_i}  \leq  1 + (3m^22^m)^{-1}$
\end{itemize}
\end{lem}
\begin{proof}
For infinite places $v \in S$, the linear form $l_{iv}'$ are uniquely determined by the permutation of $\{1, \cdots, m\}$ from Lemma \ref{lem4}. Therefore $\mathcal{T}$ is uniquely determined by $k_v ~(v \mid \infty)$ and $(c;d)$. It
follows that for $\mathcal{T}$ we have at most $$(m!)^{r_1+r_2}(m!)^{-s}(64em^42^m\epsilon^{-1})^{ms+m} \leq C'$$ possibilities. The second point follows from the definition of  $\mathcal{T}$. Namely, from Lemma \ref{lem5},  we have $$e_{iv} \leq \dfrac{12s(v)}{10\epsilon} ~~  \mbox{for} ~~ v \in S, ~~i=1, \cdots, m.$$ Moreover,
$$R_v(n)^{-1} =  \prod_{i=1}^{m}\vert l_{iv}(X_n) \vert_v  \geq R(n)^{{\displaystyle \sum_{i=1}^{m} e_{iv} }-\frac{m}{c(m)s}}$$ for $v \in S.$  Hence $$R(n)^{-1} = \prod_{v \in S}^{}R_v(n)^{-1} \geq R(n)^{{\displaystyle \sum_{i,v}^{} e_{iv}} -\frac{m}{c(m)}}$$ which gives us the desired result.   By \eqref{eq:20} and Lemma \ref{lem3}, we have $$R(n)^{d_1+\cdots + d_m} \leq \lambda_1(n) \cdots \lambda_m(n) < \Delta_K^{m/2d} R(n) \leq R(n)^{1 + \frac{1}{3m^22^m}}$$ where for the last inequality we utilized the fact $n> \tau(\epsilon)$. Moreover
$$R(n)^{d_1+\cdots + d_m} \geq \lambda_1(n) \cdots \lambda_m(n) R(n)^{- \frac{1}{4m^22^m}} >    R(n)^{1 - \frac{1}{3m^22^m}}$$ where for the last inequality we have also used that $n> \tau(\epsilon)$. Combining the two relations above complete the proof of the Lemma.
 \end{proof}
Now we  consider the vectors $v_1, \cdots, v_m$ from Lemma \ref{lem5}. We recall that $l_{iv}' = l_{k_v(i),v}$ for infinite places $v$ and  $l_{iv}' = l_{iv}$ if $v$ is a finite place in $S$. Lemma \ref{lem4}, relations \eqref{eq:19} and  \eqref{eq:20} thus yield that
\begin{equation}\label{eq:g}
\vert l_{iv}'(v_j) \vert_v  \leq   \left( (2d\vert \Delta_K\vert^{1/2})^{2m} \right)^{s(v)} R(n)^{e_{iv} + s(v) ( d_{\min(i,j)} +1/c(m)} ).
\end{equation}

Let $t$ be the largest integer such that $\lambda_t \leq1$ and $V$, the $K$-vector space generated by $\psi^{-1}(\lambda_t \Pi(n))$. We have $X_n \in V$ since $\lambda_{t+1} > 1$ and moreover  $\{v_1, \cdots, v_t\}$ is a basis of $V$. We also have $R(n)^{d_t} \leq \lambda_t \leq 1$ which implies $d_t \leq 0.$ Hence, it follows that $X_n$ lies in the $K$-vector space generated by $v_1, \cdots, v_r,$ where $r$ is the largest integer such that $d_r \leq 0.$ \\

    Below we construct a tuple $(N, \overline{\gamma}, \mathcal{L})$ depending only on $\mathcal{T}$. Since the number of possibilities of  $\mathcal{T}$ is at most $C'$, it follows that the number of such tuples is at most 
$$ (64em^42^m\epsilon^{-1})^{ms+m}.$$ Let $(l_{iv}'; ~e_{iv} ; ~d_i) \in  \mathcal{T}$. There is an integer $k$ with
\begin{equation} \label{eq:e}
1\leq k \leq m-1, \quad d_{k+1} >0, \quad d_{k+1} -d_k \geq \frac{1}{m^2}.
\end{equation}
Namely, let $r$ be the largest integer such that $d_r \leq 0$. By taking $k$ the smallest integer of the set $\{d_r, ~d_{r+1}, ~\cdots, ~d_m\}$ such that $d_{k+1} -d_k$ is maximal,  we obtain the relation \eqref{eq:e} as follows: By Lemma \ref{lem5}, it is immediate that
\begin{equation} \label{eq:f}
d_m \geq \frac{1}{m}- \dfrac{1}{3m^32^m}.
\end{equation}
Therefore
$$
d_{k+1}-d_k \geq  \frac{1}{m-r}(d_m-d_r) \geq  \frac{1}{m-r} \left( \frac{1}{m}- \dfrac{1}{3m^32^m} \right) \geq \frac{1}{m^2}
 $$
 where for the second inequality, we used \eqref{eq:f}.
Put $N= \binom{m}{k}$. As before, let $\sigma_1, \cdots, \sigma_N$ be the sequence of subset of $\{1, \cdots, m\}$ of cardinality $m-k$, ordered lexicographically. We define the set of linear forms 
$\mathcal{L} =\{ \overline{l}_{iv} ~:~~ v \in S,~~ i=1, \cdots, N \}$ with $$ \overline{l}_{iv} ~:= \alpha_{iv} l_{i_1v}' \wedge \cdots \wedge  l_{i_{m-k}v}' \quad \mbox{for} \quad v \in S, ~~ i=1, \cdots,N $$ such that $\vert \overline{l}_{iv} \vert_v =1.$ We also define 
$$ \overline{e}_{iv} = e_{i_1v} + \cdots + e_{i_{m-k}v}, \quad  \overline{d}_{i} = d_{i_1} + \cdots + d_{i_{m-k}}$$ and the tuple $\overline{\gamma} = (\gamma_{iv}, ~~v \in S, ~~~ i=1, \cdots, N)$ with 
$$\gamma_{iv} = \dfrac{\epsilon}{3m} \biggl\{ \dfrac{1}{c(m) s} +  \overline{e}_{iv} + s(v) \left( \overline{d}_{i} + \dfrac{m}{c(m)} \right)  \biggl\} \quad \mbox{for} \quad v \mid \infty, ~~ i=1, \cdots,N-1 $$ and 
$$\gamma_{Nv} = \dfrac{\epsilon}{3m} \biggl\{ \dfrac{1}{c(m) s} +  \overline{e}_{Nv} + s(v) \left( \overline{d}_{N-1} + \dfrac{m}{c(m)} \right) \biggl\} \quad \mbox{for} \quad v \mid \infty $$ and 
$$\gamma_{iv} = \dfrac{\epsilon}{3m} \biggl\{ \min \left( 0, \dfrac{1}{c(m) s} +  \overline{e}_{iv} \right) \biggl\} \quad \mbox{for} \quad v \nmid \infty, ~~ i=1, \cdots,N $$ where $c(m) = 4m^32^m.$\\

\begin{lem} \label{lem6}
We have 
 $\gamma_{iv} \leq s(v)$   for $v \in S, ~~i=1, \cdots, N$ and 
 $$\sum_{v \in S}^{}  \sum_{i=1}^{m}  \gamma_{iv}  ~\leq ~ -\frac{\epsilon}{9m^3}.$$

\end{lem}
 
 \begin{proof}
 For any finite place $v \in S$, we have $\gamma_{iv} \leq 0 \leq s(v).$ Let $v$ be  an infinite place over $K$. By Lemma \ref{lem6'}, it follows $$ \overline{e}_{iv} = e_{i_1v} + \cdots + e_{i_{m-k}v} \leq (m-k) \dfrac{12s(v)}{10\epsilon} \leq \dfrac{12ms(v)}{10\epsilon}$$ and  $$ \overline{d}_{i} = d_{i_1} + \cdots + d_{i_{m-k}} \leq d_{1} + \cdots + d_{m} -kd_1 \leq 1+ \dfrac{1}{3m^22^m} + \dfrac{12m}{10\epsilon}.$$ By inserting these relations in the expression of $\gamma_{iv}$, we get 
  $$
\begin{array}{lcl}
\gamma_{iv}  &\leq & \dfrac{\epsilon}{3m} \biggl\{ \dfrac{2+m}{c(m) } +  \dfrac{24m}{10\epsilon}  +  1+ \dfrac{1}{3m^22^m} \biggl\}s(v) \\\\
   
                                                                     &\leq&  \dfrac{\epsilon}{3} \biggl\{ 3\epsilon + 2.5  \biggl\} s(v)\\\\
                                                                     
                                                                      &\leq& s(v),\\\\
\end{array}
 $$
where we used the fact that $\epsilon < 1/ 12.$ We know that $\overline{d}_N- \overline{d}_{N-1} = d_{k+1} - d_k,$ so we obtain 
$$
\begin{array}{lcl}
{\displaystyle \sum_{v \in S}^{} \sum_{i=1}^{N} \gamma_{iv}}  &\leq & \dfrac{\epsilon}{3m} \biggl\{  {\displaystyle \sum_{v \in S}^{} \sum_{i=1}^{N} \left( \dfrac{1}{c(m)s } + \overline{e}_{iv} + s(v)\left( \overline{d}_{i} + \dfrac{m}{c(m)} \right) \right)} -(\overline{d}_N- \overline{d}_{N-1} ) \biggl\}\\\\
   
                                                                     &\leq&  \dfrac{\epsilon}{3m} \biggl\{   \dfrac{N}{c(m)} +  {\displaystyle \sum_{v \in S}^{} \sum_{i=1}^{N} \overline{e}_{iv} +   \sum_{i=1}^{N}  \overline{d}_{i} }+ \dfrac{Nm}{c(m)} -(d_{k+1}- d_{k} ) \biggl\}\\\\
                                                                     
                                                                      &\leq&\dfrac{\epsilon}{3m} \biggl\{   \dfrac{(m+1)2^m}{4m^32^m} + \dfrac{7\cdot 2^{m-1}}{12m^32^m} - \dfrac{1}{m^2}  \biggl\}\\\\
                                                                      
                                                                       &\leq& -\dfrac{ \epsilon}{9m^3},
\end{array}
 $$
 where, for the third inequality we have applied Lemma \ref{lem6'} and the equality 
 $$ \sum_{v \in S}^{} \sum_{i=1}^{N} \overline{e}_{iv}=  \binom{m-1}{m-k-1} \sum_{v \in S}^{} \sum_{i=1}^{N}e_{iv} \quad \mbox{and} \quad \sum_{i=1}^{N}  \overline{d}_{i} = \binom{m-1}{m-k-1} \sum_{i=1}^{N} d_i.$$ Hence the proof is complete.
 \end{proof}

Put $Q = \left( R(n)^{1/ \epsilon} \right)^{3m}$. We have proved $(N, ~\overline{\gamma}, ~ \mathcal{L}, Q)$ satisfies the relation \eqref{eq:14'} and  \eqref{eq:15"} of the Theorem \ref{thm2}. Hence we complete the proof of that result by showing there is a vector space $W$ for which \eqref{eq:13} holds. 
\begin{lem} \label{lem7}
Let $v_1, \cdots, v_m$ be linearly independent vectors from Lemma \ref{lema} and $W$  the $K$-vector space generated by  $\{v_1, \cdots, v_k\}$. Then 
$${\rm dim}_K W = k, \quad X_n \in W, \quad f_k(W) = V(N, ~\overline{\gamma}, ~ \mathcal{L}, Q).$$
\end{lem}

 \begin{proof}
 It is clear that ${\rm dim}_K W = k$. We have proved at page 18 that $X_n \in V$ where $V$ is the vector space generated by $\{v_1, v_2, \cdots, v_r\}$ with $r$ the largest integer such that $d_r \leq 0$. By the construction of the integer $k$ above, it is immediate that $k \geq r$ which implies $X_n \in W$. We now define the vector $$\widehat{v}_i := \widehat{v}_{\sigma_i}= v_{i_1} \wedge \cdots \wedge  v_{i_{m-k}}, \quad \mbox{where} \quad \sigma_i = \{i_1, \cdots, i_{m-k}.\}$$ 
 By definition, $f_k(W)$ has basis  $\{\widehat{v}_1, \cdots, \widehat{v}_{N-1}\}.$  Furthermore $f_k(W)$ is the $K$- vector space generated by  $\Pi:= \Pi (N, \overline{\gamma},  \mathcal{L}, Q)$. Namely, we first show that $\widehat{v}_1, 
 \cdots, \widehat{v}_{N-1} \in \Pi$ and that any vector  $\widehat{v}_0 \in \Pi$ is linearly dependent on $\widehat{v}_1, \cdots, \widehat{v}_{N-1}$. This gives us what we want. Take $v \in S$ and $i,j \in \{1, \cdots, N\}$. Suppose $\sigma_i, ~~\sigma_j$ given as above. By the Laplace rule we have $$ \vert  \overline{l}_{iv}(\widehat{v}_j) \vert_v  =  \vert \alpha_{iv} \vert_v  \vert {\rm det} (l_{pv}'(v_q))_{p\in \sigma_i,~q \in \sigma_j}\vert_v .$$ From the definition of $\overline{l}_{iv}$ and Lemma \ref{lem'} (apply with $H=1$), relation \eqref{eq:11}, we have  $\vert \alpha_{iv} \vert_v = \vert l_{i_1v}' \wedge \cdots \wedge  l_{i_{m-k}v}' \vert_v^{-1} \leq m!$ and by taking the maximum over all permutations $\kappa$ of $\sigma_j$, ones get from \eqref{eq:g} and Leibniz formula of the determinant that
 $$ \vert  \overline{l}_{iv}(\widehat{v}_j) \vert_v  \leq m!(m!)^{s(v)} \max_{\kappa} \prod_{t=1}^{m-k}  (2d \Delta_K^{1/2})^{2ms(v)} R(n)^{ e_{i_tv} + s(v) \left( d_{\min(i_t, \kappa (j_t))} + \dfrac{1}{c(m)} \right)}.$$ This expression thus yields that $$ \vert  \overline{l}_{iv}(\widehat{v}_j) \vert_v  \leq m! \left(m! (2d \Delta_K^{1/2})^{2m(m-k)}\right)^{s(v)} R(n)^{ \overline{e}_{iv} + s(v) \left( d_0 + \dfrac{m}{c(m)} \right)},$$ where 
 $$d_0 = \max_{\kappa} \sum_{i=1}^{m-k} d_{\min(i_t, \kappa(j_t))} \leq \min\{\overline{d}_i, \overline{d}_j\}.$$  Using the fact that $n> \tau(\epsilon)$, we obtain 
 $$\left(m!^2 (2d \Delta_K^{1/2})^{2m(m-1)}\right) \leq \vert \alpha_1 \vert^{\dfrac{n\epsilon}{sc(m)}} \leq  R(n)^{\dfrac{1}{sc(m)}}.$$ Hence we obtain
 \begin{equation}\label{eq:24}
\vert  \overline{l}_{iv}(\widehat{v}_j) \vert_v  \leq  R(n)^{\biggl\{ \dfrac{1}{c(m) s} +  \overline{e}_{iv} + s(v) \left( \overline{d}_{\min(i,j)} + \dfrac{m}{c(m)} \right) \biggl\}} 
\end{equation}
for $v \in S; ~~ (i,j) \in \{1, \cdots, N\} \times \{1, \cdots, N-1\}.$ Together with the fact that  $Q = \left( R(n)^{1/ \epsilon} \right)^{3m}$ and the definition of $\gamma_{iv}$, we deduce that for infinite place $v$ and $i=1, \cdots, N; ~~j=1, \cdots N-1,$ one get  $$ \vert  \overline{l}_{iv}(\widehat{v}_j) \vert_v  \leq Q^{\gamma_{iv}}.$$ We know that $v_1, \cdots, v_m \in \mathcal{O}_K^m$ and so $\widehat{v}_j  \in \mathcal{O}_K^N.$ For every finite place $v$ and $j= 1, \cdots, N-1,$ one deduce from Lemma \ref{lm}, relation \eqref{eq:7} that
$$\vert  \overline{l}_{iv}(\widehat{v}_j) \vert_v  \leq  \vert   \overline{l}_{iv} \vert_v   \vert   \widehat{v}_{j} \vert_v \leq 1.$$ Together with \eqref{eq:24}, we obtain 
$$\vert \overline{l}_{iv}(\widehat{v}_j) \vert_v  \leq  R(n)^{\min \biggl\{ 0, \dfrac{1}{c(m) s} +  \overline{e}_{Nv} \biggl\}} \leq  Q^{\gamma_{iv}}.$$ Therefore $\widehat{v}_1, \cdots, \widehat{v}_{N-1} \in \Pi .$ Take $\widehat{v}_0 \in \Pi.$ We show that $\widehat{v}_0$ is linearly dependent of $\widehat{v}_1, \cdots, \widehat{v}_{N-1}$. Fix $v \in S.$ Then 
\begin{equation}\label{eq:25}
{\rm det} ( \widehat{v}_0, \cdots, \widehat{v}_{N-1})  = {\rm det} ( \overline{l}_{1v}, \cdots, \overline{l}_{Nv})^{-1} {\rm det} (\overline{l}_{iv}(\widehat{v}_j))_{i,j}.
\end{equation}
Put $f_v =  {\rm det} (\overline{l}_{iv}(\widehat{v}_j))_{i,j}.$ By the relations \eqref{eq:9}, \eqref{eq:10} and \eqref{eq:12}, we have  
\begin{equation}\label{eq:26}
\vert {\rm det} ( \overline{l}_{1v}, \cdots, \overline{l}_{Nv})^{-1} \vert_v  \leq N!(m-1)!^N(m-1)!^{Ns(v)}.
\end{equation}
Since $v_0, \cdots, v_{N-1} \in \Pi$, it follows that
\begin{equation}\label{eq:27}
\begin{array}{lcl}
\vert f_v \vert_v  &\leq & (N!)^{s(v)}  \max_{\kappa} \vert \overline{l}_{1v}(\widehat{v}_{\kappa(0)})  \cdots \overline{l}_{Nv}(\widehat{v}_{\kappa(N-1)})  \vert_v\\\\
   
                                                                     &\leq& ( N!)^{s(v)} \cdot Q^{\gamma_{1v} + \cdots + \gamma_{Nv}}    \\\\
                                                                                                                                                                                       
\end{array}
\end{equation}
By combining  \eqref{eq:25},  \eqref{eq:26} and \eqref{eq:27} , we obtain $$ \vert {\rm det} ( \widehat{v}_0, \cdots, \widehat{v}_{N-1}) \vert_v \leq   N!(m-1)!^{N(1+s(v))} ( N!)^{s(v)} \cdot Q^{\gamma_{1v} + \cdots + \gamma_{Nv}}.$$ 
By taking the product over $v \in S$ and using Lemma \ref{lem1}, we get
$$
\begin{array}{lcl}
\prod_{v \in S}^{} \vert {\rm det} ( \widehat{v}_0, \cdots, \widehat{v}_{N-1}) \vert_v  &\leq & (N! (m!))^{N})^{(s+1)} \cdot Q^{{\displaystyle \sum_{i,v}^{} \gamma_{iv}}} \\\\
   
                                                                     &\leq& ((2^m)! (m!))^{2^m})^{(s+1)} \cdot Q^{-\dfrac{\epsilon}{9m^3}} \\\\
                                                                     
                                                                     &<& 1,
\end{array}
$$
where for the last inequality we used the relation \eqref{eq:15"} and the fact that $n> \tau(\epsilon)$. Since $ \widehat{v}_j \in \mathcal{O}_S^N$ for $j=0, \cdots, N-1,$ whence $ {\rm det} ( \widehat{v}_0, \cdots, \widehat{v}_{N-1}) $ is $S$-integer. By the product 
formula, we deduce that $$ {\rm det} ( \widehat{v}_0, \cdots, \widehat{v}_{N-1})  = 0,$$ hence the proof  is completed.
 \end{proof}

\section{Proof of Theorem \ref{thm3}}\label{sec4}
Let $l, N$ be the integers $\geq 2$. For $h=1, \cdots, l$ denote by $x_h$ the block of $N$ variables $(x_{h1}, \cdots, x_{hN})$ and $\overline{\mathbb{Q}}[x_1, \cdots, x_l]$,  the ring of polynomial in the $lN$ variables $x_{11}, \cdots, x_{lN}$ with coefficients from $\overline{\mathbb{Q}}$. We use \textbf{i} to denote a tuple of non-negative integers $i_{hj} ~:~ h=1, \cdots, l; ~~ j=1, \cdots, N.$ For such a tuple  \textbf{i} 
we define the partial derivative of $F \in \overline{\mathbb{Q}}[x_1, \cdots, x_l]$ by  $$F_{\textbf{i}} = \prod_{h=1}^{l} \prod_{j=1}^{N} \left( \frac{1}{i_{hj}!} \frac{ \partial^ {i_{hj}} }{\partial  x_{hj}^ {i_{hj}} }F \right). $$ Let 
$d:= (d_1, \cdots, d_l)$ be a tuple of positive integers and for a tuple  \textbf{i}  as above, put $$\left( \textbf{i} /d \right) = \sum_{h=1}^{l} \frac{1}{d_h} \biggl\{  i_{h1} + \cdots + i_{hN} \biggl\}.$$

Assume $F \neq 0$. Then the index of $F$ at $x$ which respect to $d$ noted ${\rm Ind}_{x,d}(F)$ is defined as the largest number $\sigma$ such that $F_{\textbf{i}} (x) =0$ for all  $\textbf{i}$ with $\left( \textbf{i} /d \right) \leq \sigma.$ It is well known (see for e.g. Chap. VI,  Lemma 4A in \cite {SB}) that $${\rm Ind}_{x,d}(FG) = {\rm Ind}_{x,d}(F) + {\rm Ind}_{x,d}(G)  \quad \mbox{for} \quad  F, ~G \in  \overline{\mathbb{Q}}[X_1, \cdots, X_l].$$
We say that $F$ is homogeneous of degree $d_h$ in $x_h$ for $h=1, \cdots, l$ if $F$ is linear combination of monomials $$ \prod_{h=1}^{l} \prod_{j=1}^{N} X_{hj}^{i_{hj}} \quad \mbox{with} \quad \sum_{j=1}^{l} i_{hj} = d_h, ~~ h=1, \cdots, l.$$   Let $\Gamma(d)$ be the set of homogeneous polynomial of degree $d_h$ in $x_h$ for $h=1, \cdots, l$  with coefficients in $\mathbb{Z}$ such that their ${\rm \gcd}$ is $1$. We define the height $H(F) := H(a_F)$, where $a_F$ is the vector of coefficients of $F$.
Let $L$ be a number field,  $v$ a place on $L$ and $F \in L[x_1, \cdots, x_l]$. Put  $\vert F \vert_v = \vert a_F \vert_v$ and  $$H(F) = \prod_{v \in M_L}^{} \vert F \vert_v.$$ From the inequality $\binom{a}{b} \leq 2^{a}$ with $a \geq b$ integers, it easy to see that  
\begin{equation}\label{eq:j}
\vert F_{\textbf{i}} \vert_v \leq  2^{(d_1+ \cdots + d_l)s(v)}\vert F\vert_v.
\end{equation}
By taking the product over all places, we obtain 
\begin{equation}\label{eq:k}
H( F_{\textbf{i}})  \leq  2^{(d_1+ \cdots + d_l)}H(F).
\end{equation}
The height of an $(N-1)$-dimensional linear subspace  of  $L^N$ defined $$V = \{x \in L^N~:~~ a_1x_1 + \cdots + a_Nx_N = 0\},$$ with $a= (a_1, \cdots, a_N) \in L^N \setminus \{0\}$ is given by $H(V) = H(a)$.\\

 We need the next lemma which is a generalisation of Roth's Lemma due to  Evertse (\cite{E}, Lemma 24)  and based on  Schmidt's result ( Chap. VI, Theorem 10B, \cite{SB}). 
\begin{lem} \label{lem8''}
Let  $d=(d_1, \cdots, d_l)$ be a tuple of positive integers and $ 0< \theta \leq 1$. Suppose that $$\dfrac{d_h}{d_{h+1}} \geq \dfrac{2l^2}{\theta}  \quad \mbox{with} \quad  h=1, \cdots, l-1.$$ Further, let $F$ be a non-zero polynomial of $\Gamma(d)$ and $V_1, \cdots, V_l$  be a  $(N-1)$-dimensional linear subspaces  of  $L^N$ with $$H(V_h)^{d_h} \geq \biggl\{ e^{d_1 + \cdots + d_l}H(F)\biggl\}^{(N-1) (3l^2 / \theta)^l}  \quad \mbox{with} \quad  h=1, \cdots, l.$$ Then there is $x_h \in V_h$ for $h=1, \cdots, l$ such that for $x=(x_1, \cdots, x_l)$, we have  ${\rm Ind}_{x,d}(F) < l\theta. $
\end{lem}
 Let $V$ be a $(N-1)$-dimensional linear subspace  of  $L^N$. A grid of size $A$ in $V$ is a set $$\Gamma = \{g_1x_1 + \cdots + g_{N-1}x_{N-1};~~ x_i \in \mathbb{Z}; ~~ \vert x_i \vert \leq A\}$$  where 
 $\{g_1, \cdots, g_{N-1}\}$ is a basis of $V.$ \\ 
 The following Lemma is proven in \cite{E} (See Lemma 26) using Schmidt's argument in the proof of Lemma 8A and 8B, Chap. VI in \cite{SB}.
\begin{lem} \label{lem9}
Let  $d=(d_1, \cdots, d_l),~F,~V_1, \cdots, V_l,~\theta$ having the same meaning of Lemma \ref{lem8''} and satisfy the condition of  Lemma \ref{lem8''}. For $h=1, \cdots, l$ let $\Gamma_h$ be any grid in $V_h$ of size $N /  \theta.$ Then there are $x_1 \in \Gamma_1, \cdots, x_l \in \Gamma_l$ such that for $x=(x_1, \cdots, x_l)$, we have  ${\rm Ind}_{x,d}(F) < 2l\theta.$
\end{lem}

In what follow,  we fix $(N, \overline{\gamma}, \mathcal{L})$ as defined in the proof of  Theorem \ref{thm2} meaning that $N=\binom{m}{t}$ with some integer $t$ and $m$ is the number of distinct roots of the characteristic polynomial $P$ of $\{u_n\}$, $K$ be the splitting field of $P$ and  $S$ be the set  containing all the prime ideals above $\alpha_i's$ and  all the archimedean  places over $K$ and $s$ its cardinality,
$ \overline{\gamma}=(\gamma_{iv}:~~v \in S;~~ i=1,\cdots N)$ with $\gamma_{iv} \leq s(v)$, by setting $\delta = \dfrac{\epsilon}{9m^3}$, we also have 
$$\sum_{i,v}^{} \gamma_{iv} \leq - \delta \quad \mbox{and} \quad \mathcal{L} =\{\overline{l}_{iv} :~~v \in S;~~ i=1,\cdots N \}$$ with $\vert\overline{l}_{iv} \vert_v =1$. We keep fixed $\Pi(Q),~~V(Q)$ for $\Pi(N, \overline{\gamma}, \mathcal{L},Q), ~~V(N, \overline{\gamma}, \mathcal{L},Q)$ respectively. We assume that $Q$ satisfy relation \eqref{eq:13'}. The next auxiliary result gives the relation between $Q$ and the height $H(V(Q))$ of $V(Q)$.

\begin{lem} \label{lem9"}
There is an $(N-1)$-dimensional linear subspace  of  $K^N$ with the following property. For every $Q$  with   \eqref{eq:13'}, we have 
$$V(Q) = V \quad \mbox{or} \quad H(V(Q)) \geq Q^{\dfrac{\delta}{3s}},$$ where $V$ is a fixed  $(N-1)$-dimensional linear subspace  of  $K^N.$
\end{lem}

 \begin{proof}
Fix $Q$  as in the relation  \eqref{eq:13'}. The first expression in relation \eqref{eq:13'} implies there are linearly independent vectors $g_1, \cdots, g_{N-1} \in \Pi(Q).$ Put $g^* = (g_1 \wedge \cdots \wedge g_{N-1})^*,$ where $x^*$ is defined at the page 6.  It clear from \eqref{eq:6} that 
$$V(Q) = \{ x \in K^N~:~~ g^* \cdot x =0\}.$$ Define the linear forms $$l_{kv}^* = \left(\overline{l}_{1v} \wedge \cdots \wedge \overline{l}_{k-1v} \wedge \overline{l}_{k+1v} \wedge \cdots \wedge \overline{l}_{Nv}\right)^*$$ and put
$$D_{kv} =  l_{kv}^* (g^*) \quad \mbox{for} \quad v \in S;~~ i=1,\cdots N.$$ By Laplace identity \eqref{eq:5} we have  $$D_{kv} =  {\rm det}(\overline{l}_{iv}(g_j)_{i\neq k,j})$$ and using the fact  that  $g_1, \cdots, g_{N-1} \in \Pi(Q)$, one deduces 
$\vert l_{iv}(g_j) \vert_v \leq Q^{\gamma_{iv}}$ for $v \in S,~~i=1, \cdots, N; ~~j=1, \cdots, N-1.$ Therefore, it follows from Leibniz formula of the determinant that
\begin{equation}\label{eq:28}
\vert D_{kv} \vert_v  \leq  (N-1)!^{s(v)} Q^{\gamma_{1v} + \cdots +\gamma_{Nv}- \gamma_{kv}}.
\end{equation}
Suppose for the moment there is a tuple $(i_v; ~~v \in S)$ with $i_v \in \{1, \cdots, N\}$  such that
\begin{equation}\label{eq:29}
D_{i_v,v} \neq 0.
\end{equation}
  and 
\begin{equation}\label{eq:30}
\sum_{v}^{} \gamma_{i_v,v} \geq - \frac{\delta}{2}.
\end{equation}
 Then, one get 
\begin{equation}\label{eq:31}
\prod_{v \in S}^{} \vert D_{i_v,v}  \vert_v \leq (N-1)! Q^{-\delta /2}.          
\end{equation} 
Fix $v \in S$ and $k:=i_v$. The product formula applied to $D_{kv}$ and the relation \eqref{eq:7} give us 
$$
\begin{array}{lcl}
1 &= & \prod_{w \in M_K}^{} \vert D_{kv}  \vert_w\\\\
   
   &\leq&  \vert D_{kv}  \vert_v \left( \prod_{w \neq v}^{} N^{s(w)} \vert  l_{kv}^* \vert_w \vert  g^* \vert_w \right) \\\\\\
                                                                     
  &\leq& \dfrac{ \vert D_{kv}  \vert_v }{ N^{s(v)} \vert  l_{kv}^* \vert_v \vert  g^* \vert_v}   \left( \prod_{w \in M_K}^{} N^{s(w)} \vert  l_{kv}^* \vert_w \vert  g^* \vert_w \right)\\\\\
  
   &\leq& \dfrac{ \vert D_{kv}  \vert_v }{ N^{s(v)-1} \vert  l_{kv}^* \vert_v \vert  g^* \vert_v}   H (l_{kv}^*) H( g^* ).
\end{array}
$$
This implies $$\vert D_{kv}  \vert_v  \geq N^{s(v)-1} \vert  l_{kv}^* \vert_v \vert  g^* \vert_v  H (l_{kv}^*)^{-1} H( g^* )^{-1}.$$ Together with  Lemma \ref{lem'} (relation \eqref{eq:11} with $H=m!$), we obtain  
$$\vert  l_{kv}^* \vert_v  \geq \frac{(m!)^{1-N}}{(N-1)!}\quad \mbox{and} \quad H (l_{kv}^*) \leq (N-1)! (m!)^{N-1}.$$  Hence $$ \prod_{v \in S}^{} \vert D_{kv}  \vert_v \geq  N(m!)^{-2s(N-1)} (N)!^{-2s} H(V(Q))^{-s} \cdot \prod_{v \in S}^{} \vert g^*  \vert_v .$$
By combining this with \eqref{eq:31}, we obtain $$H(V(Q))^{s} \geq  N(m!)^{-2s(N-1)} (N)!^{-2s-1} Q^{\delta /2}  \geq Q^{\delta /3},$$ where, for the last inequality we have used the fact that $N=\binom{m}{t} \leq 2^m$ and $$Q^{-\delta /6} < (9m!)^{-\dfrac{\delta e^{2C"}}{6}} < N(m!)^{-2s(N-1)} (N)!^{-1}$$ with $C"$ defined into the Theorem \ref{thm3}.  Now assume there is no tuple  $(i_v; ~~v \in S)$ with $i_v \in \{1, \cdots, N\}$ satisfying  \eqref{eq:29} and \eqref{eq:30}. We show there is a fixed $(N-1)$-dimensional linear subspace $V$ of $K^N,$ independent of $Q$ such that $V(Q) = V.$ Fix $v \in S$. Put $$I_v = \{i \in \{1, \cdots, N \} ~~:~~D_{iv} \neq 0\}.$$ We have 
\begin{equation}\label{eq:32}
l_{kv}^*(g^*) = 0 \quad \mbox{for} \quad v \in S, ~~ k \in \{1, \cdots, N\} \setminus  I_v          
\end{equation} 
 If  $I_v = \{1, \cdots, N\}$ for each $v \in S$, take $h=(1, 0, \cdots, 0)$ and we define  $$V := \{x \in K^N~:~~ x \cdot h =0\}.$$ Otherwise,
by the relation \eqref{eq:10} and the definition of $\overline{l}_{iv}$ in the proof of theorem \ref{thm2}, we have $$H(l_{kv}^*) \leq  (N-1)! (m!)^{N-1} \quad \mbox{for} \quad v \in S, ~~ i = 1, \cdots, N.$$ Since the rank of the set of such $l_{kv}^*$ is less than $N-1$ (by the definition of $g^*$), it follows from  Lemma \ref{lem1''} (with $H=  (N-1)! (m!)^{N-1}$) the existence of  a non-zero vector $h \in K^N$ such that
\begin{equation}\label{eq:33}
l_{kv}^*(h) = 0 \quad \mbox{for} \quad v \in S, ~~ i \in \{1, \cdots, N\}\setminus I_v          
\end{equation} 
with $$H(h) \leq (N-1)!^{N} (m!)^{(N-1)^2}.$$ Therefore, we define $$V := \{x \in K^N~:~~ x \cdot h =0\}.$$
We now want to show that $V(Q) =V.$ Since $V(Q)$
 is the vector space generated by $\Pi(Q)$ and $V(Q)$ and $V$ have dimension $N-1$, it suffices to show that $\Pi(Q) \subset V.$ Let $x \in  \Pi(Q)$, we will prove that $x\cdot h = 0$. Put $\Delta_v = {\rm det}(\overline{l}_{1v}, \cdots, \overline{l}_{Nv}).$ Using the fact that $\{\overline{l}_{1v}, \cdots, \overline{l}_{Nv}\}$ is a basis of $K^N$, we obtain $$x \cdot h = \Delta_v^{-1} \sum_{i=1}^{N} \overline{l}_{iv}(x) l_{iv}^*(h).$$ In view of the relation \eqref{eq:33}, $$x \cdot h = \Delta_v^{-1} \sum_{i\in I_v}^{} \overline{l}_{iv}(x) l_{iv}^*(h).$$ From Lemma \ref{lem1''} (the relation \eqref{eq:12} with $H=m!$), we have 
\begin{equation}\label{eq:34}
\vert \Delta_v\vert_v^{-1}  \leq N!(m!)^N \quad \mbox{for} \quad v \in S.          
\end{equation} 
Further, by applying the relation \eqref{eq:7} and \eqref{eq:9}, we get  
\begin{equation}\label{eq:h}
\vert l_{iv}^*(h) \vert_v \leq \vert \overline{l}_{1v} \wedge \cdots \wedge \overline{l}_{i-1v} \wedge \overline{l}_{i+1v} \wedge \cdots \wedge \overline{l}_{Nv} \vert_v \cdot \vert h \vert_v \leq  \left( N!(m!)^{N-1}\right)^{s(v)} \vert h \vert_v.
\end{equation}
If there exists $v \in S$ such that $I_v = \emptyset$ then, using the above equality and \eqref{eq:33}, ones deduce $x \cdots h =0$, therefore  the desired result follows. Otherwise; for $v \in S,$ choose $j_v \in I_v$ such that $\gamma_{j_v,v} = \max_{i \in I_v} \gamma_{iv}$. Since $x \in \Pi(Q)$, it follows $$\vert  \overline{l}_{iv}(x) \vert_v \leq  Q^{\gamma_{j_v,v}} \quad \mbox{for} \quad v \in S,~~ i \in I_v.$$ Combining this inequality  with   \eqref{eq:33} and  \eqref{eq:h}, we deduce 
\begin{equation}\label{eq:i}
\begin{array}{lcl}
\vert x\cdot h \vert_v &= & \vert \Delta_v\vert_v^{-1}  \big|  { \displaystyle \sum_{i\in I_v}^{} \overline{l}_{iv}(x) l_{iv}^*(h)} \big|_v \\
   
   &\leq&  N!(m!)^N \cdot N^{s(v)} \cdot  \left( N!(m!)^{N-1}\right)^{s(v)} \vert h \vert_v Q^{\gamma_{j_v,v}}.                                                                     
\end{array}
\end{equation}
Further $j_v \in I_v$ for $v \in S,$ the tuple $(j_v,~~ v\in S)$ satisfy \eqref{eq:29}, so by assumption, it does not satisfy  \eqref{eq:30}. Together with \eqref{eq:i} and the product formula, if  $\vert x\cdot h \vert_v \neq 0,$ then

$$
\begin{array}{lcl}
1 &= & \prod_{v \in S} \vert x \cdot h \vert_v \cdot   \prod_{v \notin S} \vert x \cdot h \vert_v \\\\
   
   &\leq&  (m!)^{Ns} \cdot N!^s N \cdot  \left( N!(m!)^{N-1}\right) Q^{-\delta / 2} H(h)\\\\
    
    &\leq&  (m!)^{Ns + (N-1)^2 +(N-1)} \cdot (N!)^{N+s+1}  Q^{-\delta / 2}                                                                
\end{array}
$$
which contradicts the fact that $Q> (9m!)^{e^{2C"}}$ by using the fact that $N \leq 2^m$.
 \end{proof}
We need the following gap principle.
\begin{lem} \label{lem8}
Let $a,~b$ be reals with $b>a>(9m!)^{e^{2C"}}.$ There is a collection of $(N-1)$-dimensional linear subspaces of $K^N$ of cardinality at most $$C(a,b) = 1 + 4 \delta^{-1} \log \left( \dfrac{\log b}{\log a} \right)$$ such that for every $Q$ satisfying \eqref{eq:13'} and $a \leq Q \leq b,$ the vector space $V(Q)$ belongs to this collection. 
\end{lem}
\begin{proof}
The proof is  similar to Lemma 28 in \cite{E} so let to the reader
\end{proof}

Let  us consider $F \in \Gamma(d)$. For each $v \in S,$ we put $$U_{hiv} =  \overline{l}_{iv}(x_h) \quad h=1, \cdots, l;  ~~ i=1, \cdots, N$$ and 
$$F_{\textbf{i}} = \sum_{\textbf{j}}^{} C(\textbf{i},\textbf{j},v) U_{11v}^{j_{11}} \cdots U_{lNv}^{j_{lN}}$$ where the sum is taken over tuples  $\textbf{j}$ with $$\sum_{k=1}^{N} j_{hk} = d_h - \sum_{k=1}^{N} i_{hk} \quad \mbox{for} \quad h=1, \cdots l.$$

\begin{lem} \label{lem10}
Let $\theta$ be real with $\theta < 1/N$. Assume $l > 4\theta^{-2} \log (2Ns)$ and $d=(d_1, \cdots, d_l)$ any $l$-tuple of positive integers. Then there is a polynomial $F \in \Gamma(d)$ with the following properties:
\begin{itemize}
\item  $H(F) \leq (8^Nm!)^{d_1 + \cdots + d_l}.$
\item For all tuple $v \in S$ and all $\textbf{i} ,~\textbf{j} $ with $\left( \textbf{i} /d \right) \leq 2l\theta$,  $$ C(\textbf{i}, \textbf{j},v) \neq 0 \quad \mbox{implies} \quad   \max_{k} \Bigg| \sum_{h=1}^{l} \frac{j_{hk}}{d_h} - \frac{l}{N} \bigg| < 3lN \theta.$$
\item for all tuple $\textbf{i}$,we have $$\prod_{v \in S}^{} \max_{\textbf{j}} \vert C(\textbf{i}, \textbf{j}, v)\vert_v \leq \left( 2 \cdot 16^N m! N!^{s+1}\right)^{ d_1 + \cdots + d_l}.$$
\end{itemize}
\end{lem}
\begin{proof}
By definition of $\overline{l}_{iv}$, its coefficients are rational numbers. Let $t$ be the maximal number of pairwise non-proportional linear form among  $\overline{l}_{iv}~~(v \in S~~i=1, \cdots, N).$ Then $t \leq Ns$ and we 
have $l > 4\theta^{-2} \log (2t)$. This is precisely the condition of Theorem 6A, Chap.VI in \cite{SB}. Then there exists a polynomial $F \in \Gamma(d)$ having the coefficients of ${\rm \gcd}$ equals $1.$ For 
$v \in S$ is infinite, we have $\vert \overline{l}_{iv} \vert_v \leq m!.$ Together with Theorem 6A, Chap.VI  in \cite{SB}, the first point of our Lemma is proven and by the Polynomial Theorem (Theorem 7A,  Chap.VI in \cite{SB}) we have our second point. Fix $v \in S$. Since the coefficients of  $F$ have ${\rm gcd} ~~1$, it follows that  $$ \vert F \vert_v = H(F)^{s(v)} \leq (8^N m!)^{s(v) (d_1 + \cdots + d_l)}.$$ Together with \eqref{eq:j} give us 
\begin{equation}\label{eq:34'}
 \vert F_{\textbf{i}} \vert_v   \leq   (2\cdot8^N m!)^{s(v) (d_1 + \cdots + d_l)} 
\end{equation} 

Fix $v \in S$ and consider the matrix $A_v = (\overline{l}_{1v}, \cdots, \overline{l}_{Nv})$. By Cramer's rule, we have 
\begin{equation}\label{eq:35'}
x_{hi} = \sum_{k=1}^{N} \eta_{ik} U_{hkv}  \quad \mbox{for} \quad h=1, \cdots, l;~~i=1, \cdots, N
\end{equation} 
 where $\eta_{ik} = \pm \Delta_{ik} \Delta_{v}^{-1},$ with $\Delta_{ik}$ is the determinant of the matrix  obtained by removing the $i^{th}$ row and $k^{th}$ column of $A_v$  and $\Delta_v = {\rm det}(A_v)$. Using Leibniz formula for the determinant and the fact that $\vert \overline{l}_{iv} \vert_v =1$, we obtain $\vert \Delta_{ij} \vert_v \leq (N-1)!^{s(v)}$ for $i,j = 1,~\cdots, N.$ Together with \eqref{eq:12} and the fact that the coordinates of $\overline{l}_{iv}$ are $0$ or $1$, we get 
\begin{equation}\label{eq:35}
\vert \eta_{ij} \vert_v \leq (N-1)!^{s(v)} \cdot N!.
\end{equation} 
Write $$F_{\textbf{i}}(x_1, \cdots, x_n) = \sum_{\textbf{j}}^{} P(\textbf{i}, \textbf{j})x_{11}^{j_{11}} \cdots x_{lN}^{j_{lN}},$$ where $x_i=(x_{i1}, \cdots, x_{iN})$. By inserting  \eqref{eq:35'}, we get 
$$F_{\textbf{i}}(x_1, \cdots, x_n) = \sum_{\textbf{j}}^{} P(\textbf{i}, \textbf{j}) \prod_{h=1}^{l} \prod_{k=1}^{N} \left( \sum_{j=1}^{N} \eta_{kj} U_{hjv}  \right)^{j_{hk}}.$$ Put 
$$ \vert F_{\textbf{i}} \vert_v   = \max_{\textbf{j}} \vert  P(\textbf{i}, \textbf{j}) \vert_v.$$ We finally obtain $$F_{\textbf{i}}(x_1, \cdots, x_l) = \sum_{\textbf{p}}^{} C(\textbf{i} ,\textbf{p} ,v) U_{11v}^{p_{11}} \cdots U_{lNv}^{p_{lN}},$$ where, for $v$  infinite place over $K$,  $\vert \cdot \vert_v^{1/s(v)}$ is the usual absolute value and so  

$$
\begin{array}{lcl}
\vert C(\textbf{i}, \textbf{p},v)  \vert_v^{\frac{1}{s(v)}} &\leq& \bigg| {\displaystyle \sum_{\textbf{j}}^{} P(\textbf{i}, \textbf{j}) \prod_{h=1}^{l} \prod_{k=1}^{N} \left( \sum_{j=1}^{N} \eta_{kj}  \right)^{j_{hk}} } \bigg|_v^{\frac{1}{s(v)}}\\\\
   
   &\leq& {\displaystyle \sum_{\textbf{j}}^{}  \vert F_{\textbf{i}} \vert_v^{\frac{1}{s(v)}}  \Bigg|  \prod_{h=1}^{l} \prod_{k=1}^{N} \left( \sum_{j=1}^{N} \eta_{kj}  \right)^{j_{hk}} } \bigg| _v^{\frac{1}{s(v)}} \\\\
    
    &\leq& {\displaystyle \sum_{\textbf{j}}^{}  \vert F_{\textbf{i}} \vert_v^{\frac{1}{s(v)}}    \prod_{h=1}^{l} \prod_{k=1}^{N} \left(N! N!^{\frac{1}{s(v)}} \right)^{j_{hk}}}.                                                                
\end{array}
$$
Together with \eqref{eq:34'}, we deduce $$\vert C(\textbf{i},\textbf{p},v)  \vert_v \leq \left( 2 \cdot 16^N m! \right) ^{s(v) (d_1 + \cdot + d_l)} \left( N!^{s(v)} (N!) \right)^{d_1 + \cdots + d_l}$$ for any $v \in S$. By taking the product over $v \in S$, we finally get 

$$
\begin{array}{lcl}
\prod_{v \in S}^{} \max_{\textbf{p}} \vert C(\textbf{i}, \textbf{p}, v)\vert_v &\leq& \left( 2 \cdot 16^N m! \right) ^{ d_1 + \cdot + d_l} \left( N!^{s+1}  \right)^{d_1 + \cdots + d_l} \\\\
   
                                                                       &\leq&  \left( 2 \cdot 16^N  m!N!^{s+1}\right)^{ d_1 + \cdots + d_l}.

\end{array}
$$
 
\end{proof}

Now we present the last step of the proof of Theorem \ref{thm3}.
 Let $(N, \overline{\gamma}, \mathcal{L})$ as constructed in the proof of  Theorem \ref{thm2} meaning that $N=\binom{m}{t}$ with some integer $t$ and $m$ is the number of distinct roots of the characteristic polynomial of $\{u_n\}$ and  $ \overline{\gamma}=(\gamma_{iv}:~~v \in S;~~ i=1,\cdots N)$ with $\gamma_{iv} \leq s(v)$, by setting $\delta = \dfrac{\epsilon}{9m^3}$, we also have 
$$\sum_{i,v}^{} \gamma_{iv} \leq - \delta \quad \mbox{and} \quad \mathcal{L} =\{\overline{l}_{iv} :~~v \in S;~~ i=1,\cdots N \}$$ with $\vert\overline{l}_{iv} \vert_v =1$.  Put $\theta = \delta/ (30N^3)$ and we recall that $C" =  2^{11m+34}s^2 \epsilon^{-4} m^{12}$. Let $l$ be the smallest integer such that $l> 4 \theta^{-2} \log (2Ns)$. So $l < 5 \theta^{-2} \log (2Ns)$. We assume that the number of subspaces $V(Q)$ satisfying  \eqref{eq:13'} has cardinality $> C"$ and will derive a contradiction from that.  By denoting $\lfloor \cdot \rfloor$ the floor function, this number of subspaces is more than 
$$2+ (l-1)t, \quad \mbox{with} \quad t= 2 + \lfloor 4 \delta^{-1} \log(4l^2 \theta^{-1}) \rfloor $$ since, using the fact that $N<2^m$ and $l< 5 \theta^{-2} \log (2Ns)$, we get $2+ (l-1)t < C"$. By Lemma \ref{lem9"}, there are real $Q_1',~ \cdots,~ Q_{1+(l-1)t}'$ with  the relation \eqref{eq:13'} and 
$$ Q_1' <  Q_2' <  \cdots  < Q_{1+(l-1)t}' $$  such that the spaces  $V(Q_1'), \cdots, V(Q_{1+(l-1)t}')$ are all distinct and different from $V$, therefore 
\begin{equation}\label{eq:l}
H(V(Q_i')) \geq Q_i'^{\dfrac{\delta}{3s}} \quad \mbox{for} \quad i=1, \cdots, 1+t(l-1).
\end{equation}
We set $$Q_1 =Q_1',~ Q_2=Q_{1+t}',~ \cdots,~ Q_l = Q_{1+t(l-1)}'.$$ and 
$$V_h = V(Q_h) \quad \mbox{for} \quad h= 1, \cdots, l.$$ There are $t$ different spaces $V(Q)$ with $Q_h \leq Q \leq Q_{h+1}$. By applying Lemma \ref{lem8} with $a=Q_{h+1}$ and $b=Q_h$, we obtain  $$Q_{h+1} \geq Q_{h}^{4l^2 \theta^{-1}} \quad \mbox{for} \quad h= 1, \cdots, l-1.$$ We define  $$d_1 = 1 + \Bigl\lfloor \frac{\log Q_l}{\theta \log Q_1} \Bigr\rfloor.$$ Let us define $d_h$ as an integer satisfying the relation 
\begin{equation}\label{eq:36}
d_1\log Q_1 \leq d_h \log Q_h <(1+ \theta) d_1 \log Q_1  \quad \mbox{for} \quad h= 2, \cdots, l.
\end{equation}
Furthermore, we may take $$d_h = 1 + \Bigl\lfloor \frac{d_1 \log Q_1}{ \log Q_h} \Bigr\rfloor \quad \mbox{for} \quad h= 2, \cdots, l.$$
In what follows, we denote by $F$ the polynomial obtained from Lemma \ref{lem10}. In view of the definition of $d_h$,  one gets $$\dfrac{d_h}{d_{h+1}} =  \dfrac{d_h \log Q_h}{d_{h+1} \log Q_{h+1}} \cdot \dfrac{\log Q_{h+1}}{\log Q_h} \geq \dfrac{2l^2}{\theta}.$$  From \eqref{eq:l}, we deduce  $$H(V(Q_h))^{d_h} \geq Q_h^{\frac{d_h \delta}{3s}} > Q_1^{\frac{d_1 \delta}{3s}}.$$ On the other hand, $d_1 + \cdots + d_l  \leq ld_1$ and  by Lemma \ref{lem10}, we infer

$$
\begin{array}{lcl}
\biggl\{ e^{d_1 + \cdots + d_l}H(F)\biggl\}^{(N-1) (3l^2 / \theta)^l} &\leq& \left( e^{ld_1}(8^Nm!)^{ld_1} \right) ^{(N-1) (3l^2 / \theta)^l}  \\\\
   
                                                                       &\leq& \left( e8^Nm! \right) ^{ld_1(N-1) (3l^2 / \theta)^l}    \\\\
                                                                       
                                                                        &\leq& \left( 9m! \right) ^{ld_1N^2 (3l^2 / \theta)^l}  \\\\
                                                                        
                                                                        &<& Q_1^{\dfrac{d_1 \delta}{3s}} 
                                                                                                                                   
\end{array}
$$
where for the last inequality, we used the fact that $l< 5 \theta^{-2} \log (2Ns)$ and  $Q_1 > (9m!)^{e^{C"}}$. Therefore   $$H(V_h)^{d_h} \geq \biggl\{ e^{d_1 + \cdots + d_l}H(F)\biggl\}^{(N-1) (3l^2 / \theta)^l}$$ and so  $l, N, \theta,  d_1, \cdots, d_l, F, V_1, \cdots, V_l$ satisfy the  conditions of Lemma \ref{lem8''}.  For $h = 1, \cdots, l,$ choose a linearly independent set of vectors $$\{g_{h1}, \cdots, g_{hN-1}\}$$ from $\Pi(Q_h)$ and let $\Gamma_h$ be the grid of size $N\theta^{-1},$ given by   $$\Gamma_h = \{g_{h1}u_1 + \cdots + g_{hN-1}u_{N-1};~~ u_i \in \mathbb{Z}; ~~ \vert u_i \vert \leq N\theta^{-1}\}.$$ By Lemma \ref{lem9},   there are $x_1 \in \Gamma_1, \cdots, x_l \in \Gamma_l$ such that for $x=(x_1, \cdots, x_l)$, we have  ${\rm Ind}_{x,d}(F) < 2l\theta$ and so  
\begin{equation}\label{eq:m}
f:= F_{\textbf{i}}(x_1, \cdots, x_l) \neq 0,
\end{equation}
 for some $\textbf{i} = (i_{11}, \cdots, i_{lN})$ with $(\textbf{i}/d)= {\rm Ind}_{x,d}(F).$ Since $g_{hj} \in \Pi(Q_h)$, it follows that $x_h \in \mathcal{O}_S^N$ for $h=1, \cdots, N.$ Further, the fact that $ F_{\textbf{i}}$ has coefficients in $\mathbb{Z}$ and the product formula yield $\prod_{v \in S}^{} \vert f \vert_v \geq 1.$ Below we show that 
$\prod_{v \in S}^{} \vert f \vert_v  < 1.$ Fix $v \in S.$ Put $u_{hiv} = \overline{l}_{iv}(x_h)$ for  $h=1, \cdots, l; ~~ i=1, \cdots, N.$ Since  $x_h \in \Gamma_h$  of size $N \theta^{-1}$, we have 
\begin{equation}\label{eq:39}
\vert u_{hiv} \vert_v \leq (N^2 \theta^{-1})^{s(v)} Q_h^{\gamma_{iv}} \quad \mbox{for} \quad h=1, \cdots, l; ~~ i=1, \cdots, N.
\end{equation}
From Lemma \ref{lem10} we may write 
\begin{equation}\label{eq:40}
f = \sum_{\textbf{j}}^{} C(\textbf{i},\textbf{j},v) u_{11v}^{p_{11}} \cdots u_{lNv}^{j_{lN}}
\end{equation}
with 
\begin{equation}\label{eq:41}
 \max_{k} \Bigg| \sum_{h=1}^{l} \frac{j_{hk}}{d_h} - \frac{l}{N} \bigg| < 3lN \theta \quad \mbox{for}\quad  C(\textbf{i},\textbf{j},v) \neq 0.
\end{equation}
The number of possibilities  of  $\textbf{j}$ is at most $(2^N)^{(d_1 + \cdots + d_l)}$ and we know that $d_1 + \cdots + d_l \leq ld_1.$ Thus, we have 
$$
\begin{array}{lcl}
\vert f \vert_v &\leq& (2^N)^{s(v) (d_1 + \cdots + d_l)} \cdot \max_{\textbf{j}} \vert C(\textbf{i},\textbf{j},v) \vert_v \vert u_{11v}\vert_v^{j_{11}} \cdots \vert u_{lNv}\vert_v^{j_{lN}} \\\\
   
                                                                       &\leq&  (2^NN^2\theta^{-1})^{s(v) (d_1 + \cdots + d_l)} \max_{\textbf{j}} \vert C(\textbf{i},\textbf{j},v) \vert_v Q_1^{{\displaystyle \sum_{k}^{} \gamma_{kv} j_{1k}}} \cdots  Q_l^{{\displaystyle \sum_{k}^{} \gamma_{kv}j_{lk}}}

\end{array}
$$
where for the last inequality, we used \eqref{eq:39}. However
\begin{equation}\label{eq:o}
\begin{array}{lcl}
{\displaystyle \sum_{h, k}^{} \gamma_{kv} \dfrac{j_{hk}}{d_h} \cdot \dfrac{d_h \log Q_h}{d_1 \log Q_1}}&\leq& {\displaystyle \sum_{k=1}^{N}  [ (\gamma_{kv}-s(v)) + s(v)] \sum_{h=1}^{l}  \dfrac{j_{hk}}{d_h}  \cdot \dfrac{d_h \log Q_h}{d_1 \log Q_1} }\\\\

                                                                       &\leq& (1+\theta) \left( {\displaystyle \sum_{k=1}^{N}  l\gamma_{kv} \left( \frac{1}{N} - 3N\theta \right) +  
                                              s(v)(3lN^2 \theta - l)  + s(v)(l+3lN^2\theta)} \right) \\\\
                                                                       
                                                                       &\leq& l (1+\theta) \left(6N^2\theta s(v)+ \left(\frac{1}{N} - 3N\theta \right) {\displaystyle  \sum_{k=1}^{N}  \gamma_{kv}} \right).

\end{array}
\end{equation}

 From the equality  $$Q_h^{{\displaystyle \sum_{k}^{} \gamma_{kv} j_{hk}}} =  Q_1^{\dfrac{\log Q_h}{\log Q_1} {\displaystyle \sum_{k}^{} \gamma_{kv} j_{hk}}}$$ together with \eqref{eq:40} and   \eqref{eq:41}, it follows that

$$
\begin{array}{lcl}
\prod_{v \in S}^{} \vert f \vert_v  &\leq&  (2^NN^2\theta^{-1})^{(ld_1)} {\displaystyle \prod_{v \in S} \max_{\textbf{j}} \vert C(i,\textbf{j},v) \vert_v} \cdot Q_1^{ld_1(1+\theta) \left(6N^2\theta + \left(\dfrac{1}{N} - 3N\theta \right) {\displaystyle  \sum_{k,v}^{}  \gamma_{kv}} \right)} \\\\

                                                                       &\leq& (2^{5N+1}N^2\theta^{-1} m! (N!)^{s+1})^{ld_1} \cdot Q_1^ {1.25 ld_1 \left( \dfrac{\delta^2}{10N^2}+ \dfrac{\delta}{5N} -  \dfrac{\delta}{N} \right)}\\\\
                                                                       
                                                                       &\leq&  (2^{5N+1}N^2\theta^{-1} m! (N!)^{s+1})^{ld_1} \cdot Q_1^ {-3ld_1\delta/5N}\\\\
                                                                       
                                                                        &\leq&  \left( 2^{5N+1}N^2\theta^{-1} m! (N!)^{s+1} \cdot Q_1^{- 3\delta/5N} \right) ^ {ld_1}\\\\
                                                                        
                                                                         &<& 1

\end{array}
$$
where for the last inequality, we used the fact that $N \leq 2^m$ and $Q_1 > (9m!)^{e^{2C"}},~~ \theta = \delta / 30N^3$ and ${\displaystyle \sum_{k,v}^{} \gamma_{kv} \leq - \delta} $ with $\delta = \epsilon/9m^3$ and $C" =  2^{11m+34}s^2 \epsilon^{-4} m^{12}$. Hence the proof of Theorem \ref{thm3} is complete.
\\\\

\noindent \textbf{Acknowledgement}. 
The author would like to thank Prof. Pieter Moree for his valuable suggestions on the first version of this paper, as well as Prof. Yann Bugeaud for the clarification on how to refine the bounds of Theorems \ref {thm} and \ref {thm1} in the previous draft.

\bibliographystyle{plain}

\end{document}